\documentclass[11pt]{article}%
\usepackage[font=small,format=plain,labelfont=bf,up]{caption}
\usepackage[a4paper,nohead,ignorefoot,%
twoside,scale=0.85,bindingoffset=1.25cm]{geometry}
\usepackage{amssymb,amsfonts,amsmath,amsthm}
\usepackage[]{graphicx}
\usepackage{color}
\newcommand{\url}[1]{#1} 
\usepackage{hyperref}%
\definecolor{gray}{rgb}{0.2,0.2,.2}
\hypersetup{%
  unicode=true,          
  colorlinks=true,       
  linkcolor=gray,          
  citecolor=gray      
}
\definecolor{colorBlue}{rgb}{0.,.0,0.825}
\definecolor{colorGreen}{rgb}{0.,0.75,.0}
\definecolor{colorRed}{rgb}{0.99,0.,.0}

\DeclareMathOperator{\sgn}{\mathrm{sgn}}

\newcommand{\Rset}{{\mathbb{R}}}

\newcommand{\ocinterval}[2]{(#1,\,#2]}%
\newcommand{\cointerval}[2]{[#1,\,#2)}%
\newcommand{\oointerval}[2]{(#1,\,#2)}%
\newcommand{\ccinterval}[2]{[#1,\,#2]}%

\newcommand{\ini}{{\rm ini}}

\newlength{\mhpicDwidth}
\newlength{\mhpicDvsep}
\newlength{\mhpicDhsep}

\newlength{\mhpicPwidth}
\newlength{\mhpicPvsep}
\newlength{\mhpicPhsep}

\newlength{\mhpicWhsep}

\newcommand{\pair}[2]{{\left({#1},\,{#2}\right)}}

\newcommand{\bpair}[2]{{\big({#1},\,{#2}\big)}}

\newcommand{\at}[1]{{\left({#1}\right)}}
\newcommand{\nat}[1]{(#1)}
\newcommand{\bat}[1]{{\big(#1\big)}}
\newcommand{\Bat}[1]{{\Big(#1\Big)}}

\newcommand{\ul}[1]{\underline{#1}}

\newcommand{\D}{\displaystyle}

\newcommand{\bigpar}{\par\quad\newline\noindent}

\newcommand{\bjump}[1]{{\big|\!\big[#1\big]\!\big|}}

\newcommand{\abs}[1]{\left|{#1}\right|}

\newcommand{\nabs}[1]{|{#1}|}

\newcommand{\dint}[1]{\,\mathrm{d}#1}

\newcommand{\al}{{\alpha}}
\newcommand{\be}{{\beta}}

\newcommand{\ka}{{\kappa}}
\newcommand{\la}{{\lambda}}
\newcommand{\si}{{\sigma}}

\theoremstyle{plain}
\newtheorem{theorem}             {Theorem}[]

\newtheorem{lemma}      [theorem]{Lemma}

\newtheorem*{result*}{Main result}
\theoremstyle{definition}

\newtheorem*{remarks*}{Remarks}
\newtheorem*{remark*}{Remark}
\usepackage{float}
\begin{document}
%
%
%
%
\title{Traveling phase interfaces in viscous \\ forward-backward diffusion equations}
\date{\today}
\author{%
Carina Geldhauser\and
Michael Herrmann\and
Dirk Jan{\ss}en
 }
\maketitle
%
%
%
\begin{abstract}
The viscous regularization of an ill-posed diffusion equation with bistable nonlinearity predicts a hysteretic behavior of dynamical phase transitions but a complete mathematical \mbox{understanding} of the intricate multiscale evolution is still missing. We shed light on the fine structure of \mbox{propagating} phase boundaries by carefully examining traveling wave solutions in a special case. Assuming a trilinear constitutive relation we characterize all waves that possess a monotone \mbox{profile} and connect the two phases by a single interface of positive width. We further study the two sharp-interface regimes related to either vanishing viscosity or the bilinear limit.
\end{abstract}
%
%
%
\quad\newline\noindent%
\begin{minipage}[t]{0.15\textwidth}%
Keywords:
\end{minipage}%
\begin{minipage}[t]{0.8\textwidth}%
\emph{hysteretic phase interfaces, ill-posed diffusion equations, \\viscous regularization, traveling waves in piecewise linear systems}
\end{minipage}%
\medskip
\newline\noindent
\begin{minipage}[t]{0.15\textwidth}%
MSC (2010): %
\end{minipage}%
\begin{minipage}[t]{0.8\textwidth}%
35R25,  
35C07,  
35R35,  
74N30 
\end{minipage}%
%
%
%
%
%
\section{Introduction}\label{secIntro}
%
%
This work concerns special solutions to the viscous diffusion equation
\begin{align}
\label{PDE}
\partial_t u\pair{t}{x}-\nu^2\partial_x^2\partial_t u\pair{t}{x}=\partial_x^2
\Phi^\prime\bat{u\pair{t}{x}}
\end{align}
with small \emph{viscosity} parameter $\nu>0$, which can be regarded as a singular perturbation of the nonlinear diffusion equation
\begin{align}
\label{bulk}
\partial_t u\pair{t}{x}&=\partial_x^2
\Phi^\prime\bat{u\pair{t}{x}} \,.
\end{align}
We are solely interested in the bistable case, in which $\Phi^\prime$ is the derivative of a double-well potential and admits two increasing (or \emph{stable}) branches  that are separated by a decreasing (or \emph{unstable}) one. Such nonlinearities arise in the context of phase transitions while the monostable case has applications in image processing and population dynamics, see for instance \cite{Ell85,PM90,HPO04,Pad98}. It has been shown in \cite{NCP91} that the initial value problem to \eqref{PDE} is well-posed for any $\nu>0$ and that there exists a plethora of stable and unstable steady states. The limit PDE \eqref{bulk}, however, is ill-posed and often called a \emph{forward-backward parabolic equation}. Besides of \eqref{PDE}, many other regularizations of \eqref{bulk} have also been studied such as the lattice approximation and the Cahn-Hilliard equation given in \eqref{lattice}  and \eqref{CahnHilliard} below.
\par
From a mathematical point of view it is very natural to study the vanishing viscosity limit $\nu\to0$ of solutions to \eqref{PDE}, to characterize all possible limit functions, and to understand in which sense they satisfy the nonlinear equation \eqref{bulk}.  Of particular interest is the sharp-interface limit in which there exist finitely many interface curves  which separate regions in which $u$ takes values in either one of the convex components of $\Phi$. The effective dynamics for $\nu=0$ is then governed by a free boundary problem,  that complements the bulk diffusion \eqref{bulk} between any two interfaces with certain conditions on the interfaces. These equations are stated below in $\eqref{lmSC}$+$\eqref{lmFR}$ for the nonlinearity \eqref{trilin}. 
\par
For the viscous approximation we still have no complete theory for the passage to the limit $\nu\to0$ although there exist important partial results as explained below in greater detail. It has already been observed in \cite{Plo94,EP04} that the sharp-interface limit of the viscous regularization \eqref{PDE} is quite involved since the interface value of $\Phi^\prime\at{u}$ depends in a hysteretic manner on both the instantaneous state of the system and the current propagation direction of the interface. In particular, in the viscous case we find both standing and moving phase interfaces but numerical simulations as well as heuristic arguments indicate that the dynamics of the latter is rather complicated and related to several time and space scales.
\par
In this paper we prove the existence and uniqueness of monotone traveling waves solutions to \eqref{PDE} in the case that $\Phi^\prime$ is a trilinear function. These special solutions can be viewed as particular instances of moving phase interfaces for $\nu>0$ and exemplify the subtle interplay between the forward diffusion in the bulk and the backward diffusion inside a phase interface of small but positive width. Moreover, numerical simulations as presented below indicate that the moving interfaces that emerge from certain classes of initial data can be approximated by traveling waves.
%
%
\subsection{Setting and overview of previous results}
%
%
\paragraph{Simplified setting }
To obtain explicit formulas we restrict our analysis to the continuous and piecewise linear constitutive relation
\begin{align}
\label{trilin}
\Phi^\prime\at{u}=u-\sgn_\ka\at{u}=
\left\{\begin{array}{rl}
u+1&\text{for $u\in\ocinterval{-\infty}{-\ka}$\,,}\\
\D-\frac{\,1-\ka\,}{\ka}\,u&\text{for $u\in\ccinterval{-\ka}{+\ka}$\,,}\\
u-1&\text{for $u\in\cointerval{+\ka}{+\infty}$\,,}
\\\end{array}\right.
\end{align}
which involves the additional parameter $0<\ka<1$ and is illustrated in Figure \ref{FigTrilinear}. The spinodal region, in which the backward character of the equation manifests, is given by the interval $\ccinterval{-\ka}{-\ka}$, while the regions $\ocinterval{-\infty}{-\ka}$ and $\cointerval{+\ka}{+\infty}$ represent the two phases and are related to forward diffusion. In the limit as $\kappa \to 0$, the spinodal region contracts to the point $u=0$, so $\sgn_\ka$ coincides with the usual sign function and $\Phi^\prime$ becomes bilinear and discontinuous. Piecewise linear functions $\Phi^\prime$ have already been studied previously in the context of ill-posed diffusion equations since they simplify the nonlinear problem considerably.
%
%
\begin{figure}[ht!]%
\centering{%
\includegraphics[width=0.8\textwidth]{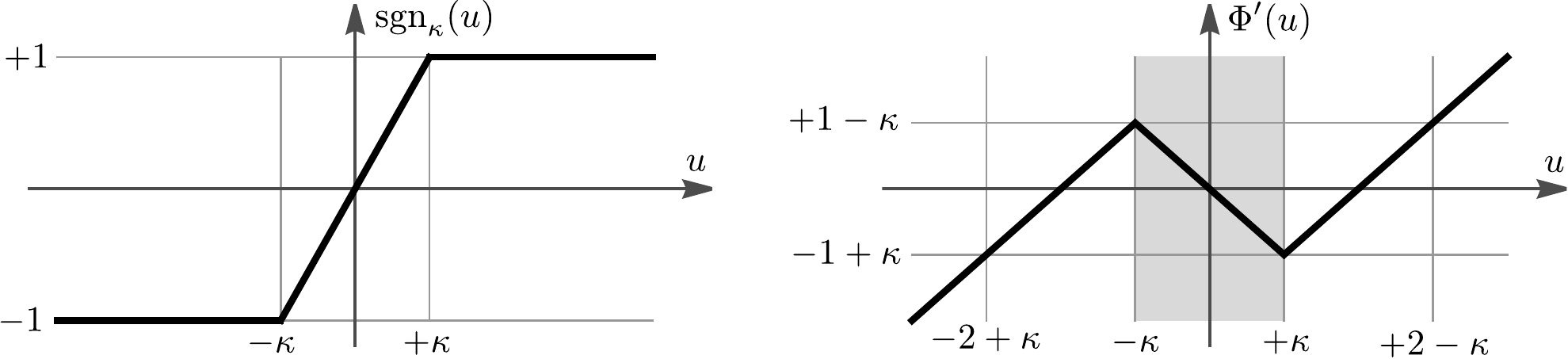}%
}%
\caption{%
\emph{Left}. The modified sign function $\sgn_\ka$ with parameter  $0<\ka<1$.
\emph{Right}. The trilinear function from \eqref{trilin} which decreases inside the spinodal interval (gray box) but increases outside. 
}%
\label{FigTrilinear}%
\end{figure}%
%
\par 
Of particular interest are \emph{single-interface solutions} of the viscous model \eqref{PDE} which exhibit a single transition between the two phases and satisfy
\begin{align}
\label{singInt}
u\pair{t}{x}\in\left\{\begin{array}{rl}
\ocinterval{-\infty}{-\ka}&\text{for $x\in\ocinterval{-\infty}{\xi_-\at{t}}$}\\
\ccinterval{-\ka}{+\ka}&\text{for $x\in\ccinterval{\xi_-\at{t}}{\xi_+\at{t}}$}\\
\cointerval{+\ka}{+\infty}&\text{for $x\in\cointerval{\xi_+\at{t}}{+\infty}$}
\\\end{array}\right.
\end{align}
for two time-dependent positions $\xi_-\at{t}\leq\xi_+\at{t}$ that represent the boundaries of the phase interface.  In the literature, such solutions are sometimes called \emph{two-phase solutions} and the monotone traveling waves constructed below fall into this class. A priori, it is not  clear that single-interface states are invariant under the dynamics of the viscous PDE \eqref{PDE} and the question of persistence has, to the best of our knowledge, not yet been conclusively answered in the general case. There exist, however, numerical evidence and related results in the limiting case $\nu=0$ as discussed below. Moreover, \cite{Jan23,HJ23a,HJ23b} establish rigorous persistence results in the special case $\ka=0$ for suitable time-discretizations of \eqref{PDE} and \eqref{bulk}.
%
%
\paragraph{On the sharp-interface limit}
Due to numerical simulations and heuristic arguments we expect that smooth single-interface solutions to \eqref{PDE}$+$\eqref{trilin} converge as $\nu\to0$  to a function $u$ which do not enter the spinodal region but possess a single phase interface of vanishing width. This means
\begin{align}
\notag
\xi\at{t}=\xi_-\at{t}=\xi_+\at{t}\,,\qquad u\pair{t}{x}<-\ka\quad\text{for}\quad x<\xi\at{t}\,,\qquad 
u\pair{t}{x}>+\ka\quad\text{for}\quad x>\xi\at{t}
\end{align}
and simplifies the bilinear constitutive relation \eqref{trilin} to 
\begin{align}
\label{simpCR}
 p\pair{t}{x}:=\Phi^\prime\bat{u\pair{t}{x}}=u\pair{t}{x}-\sgn\bat{x-\xi\at{t}}\,.
\end{align}
Moreover, the formal asymptotic analysis in \cite{EP04} predicts that $u$ and $p$ are discontinuous and continuous at the interface $x=\xi\at{t}$, respectively, and that the limit dynamics are governed by a free boundary problem with strong hysteresis that permits both standing and moving phase interfaces. The first two ingredients to the limit model are the \emph{bulk diffusion}
\begin{align}
\label{lmBD}
\partial_t u\pair{t}{x} = \partial_x^2 p\pair{t}{x}
\qquad \text{for}\quad x\neq\xi\at{t}
\end{align}
and the \emph{Stefan condition}
\begin{align}
\label{lmSC}
\bjump{p\pair{t}{\cdot}}_{x=\xi\at{t}}=0\,,\qquad\quad
\bjump{u\pair{t}{\cdot}}_{x=\xi\at{t}}=2\,,\qquad\quad
2\,\dot{\xi}\at{t}+\bjump{\partial_xp\pair{t}{\cdot}}_{x=\xi\at{t}}=0\,,
\end{align}
where we denote spatial jumps as
\begin{align*}
\bjump{p\pair{t}{\cdot}}_{x=\xi\at{t}}= \lim_{x \searrow\xi\at{t}} p\pair{t}{x} - \lim_{x \nearrow \xi\at{t}} p\pair{t}{x}.
\end{align*}
 The bulk diffusion is actually linear since \eqref{simpCR} ensures $\partial_x^2 p=\partial_x^2u$ outside the interface and the  Stefan condition \eqref{lmSC} ensures that the parabolic PDE \eqref{lmBD} holds across the interface $x=\xi\at{t}$ in a distributional sense. The last part of the limiting free boundary problem  is the \emph{hysteretic flow rule}
\begin{align}
\label{lmFR}
\begin{split}
\begin{array}{lcccc}
\text{(left moving)}&\quad&\dot{\xi}\at{t}<0&\quad\Longrightarrow\quad& p\bpair{t}{\xi\at{t}}=+1-\ka
\,,\\%
\text{(right moving)}&&\dot{\xi}\at{t}>0&\quad\Longrightarrow\quad& p\bpair{t}{\xi\at{t}}=-1+\ka
\,,\\%
\text{(standing)}&&p\bpair{t}{\xi\at{t}}\in\oointerval{-1+\ka}{+1-\ka}&\quad\Longrightarrow\quad&\dot{\xi}\at{t}=0\,,
\end{array}
\end{split}
\end{align}
which combines three dynamical relations between the current interface value of $p$ and the current interface speed $\dot\xi\at{t}$. It stipulates that the interface can only move to the left or to the right when $p\bpair{t}{\xi\at{t}}$ attains the corresponding critical value. Moreover, it is equivalent to a certain family of entropy inequalities as discussed in \cite{Plo94,EP04}.
%
%
\paragraph{Known results and open problems} 
Existence and uniqueness results for the free \mbox{boundary} problem \eqref{lmBD}$+$\eqref{lmSC}$+$\eqref{lmFR} --- or equivalently, for entropy solutions to \eqref{lmBD}$+$\eqref{lmSC} --- can be found in \cite{MTT09,Sma10,Ter11,ST13a} but the global existence is guaranteed for standing interfaces only. The uniqueness of solutions has also been shown in \cite{HH13} using arguments from \cite{Hil89,Vis06} and the corresponding Riemann problem has been solved in \cite{GT10,LM12}.
\par
For the viscous approximation \eqref{PDE}, a rigorous justification of \eqref{lmBD}$+$\eqref{lmSC}$+$\eqref{lmFR} in the limit $\nu\to0$ has not yet been accomplished although there exist partial results as in \cite{EP04,Sma10}. Moreover, \cite{HH13,HH18} derive these equations as scaling limit of the lattice approximation 
\begin{align}
\label{lattice}
\partial_t u\pair{t}{x}=\Phi^\prime\bat{u\pair{t}{x+\nu}}+\Phi^\prime\bat{u\pair{t}{x-\nu}}-2\, \Phi^\prime\bat{u\pair{t}{x}}
\end{align}
with bilinear or trilinear $\Phi^\prime$ and for a special class of initial data that produces both moving and standing interfaces but excludes any change of the propagation direction. For a slightly different lattice approximation of ill-posed diffusion equations, the limit $\nu\to0$ has been taken  for a cubic nonlinearity
in \cite{GN11,BGN13} using initial conditions that lead to finitely many standing interfaces. These results also allow to study the limit $t\to\infty$ and have later been generalized in \cite{BGN13} to a wider class of initial data.
\par
The major problem in the mathematical analysis for $\nu>0$ is that the propagation of phase interfaces involves different time and space scales since the backward diffusion inside the spinodal region produces strong fluctuations. More precisely, both the spatial and temporal derivatives of $u$ are rather large inside and near the phase interface while we expect them to be uniformly bounded in the main part of the bulk. It remains a challenging task to derive suitable localization estimates for $\partial_t u$ as well as $\partial _x u$ and to control the macroscopic impact of the microscopic fluctuations in the limit $\nu\to0$.  We believe that the exact traveling wave solutions constructed below provide a better understanding of the fine structure of propagating phase boundaries in the viscous PDE \eqref{PDE}.  
\par
Besides of single-interface solutions as in \eqref{singInt} one might also study two-phase solutions with multiple phase boundary or more general solutions in which the function $u$ is replaced by a Young measure that describes a mixture between the two phases. We refer to \cite{Plo93} for a discussion of the different solution concepts, to \cite{Hol83,ST10,ST13b,Ter14} for existence and non-uniqueness results concerning measure-valued solutions, and to \cite{Ter15} for the existence of weak solutions that penetrate the spinodal region.
%
\paragraph{Hysteresis and other models}
The dynamics of sharp interfaces depend on the \mbox{regularization} of the ill-posed diffusion equation \eqref{bulk}. Any reasonable limit model combines the bulk diffusion  \eqref{lmBD} with the Stefan condition \eqref{lmSC} but these equations do not determine the evolution of the free boundary completely and must be complemented by an extra condition. The latter takes into account microscopic details in the vicinity of the interface and differs for distinct regularization. 
\par
For the one-dimensional Cahn-Hilliard equation 
\begin{align}
\label{CahnHilliard}
\partial_t u\pair{t}{x}=\partial_x^2
\Phi^\prime\bat{u\pair{t}{x}}-\nu^2\partial_x^4 u\pair{t}{x}
\end{align}
which is discussed, among others, in \cite{BBMN11,BFG06}, there is not hysteresis in the limit $\nu\to0$ since the interface value of $p$ is uniquely determined by the properties of $\Phi$. More precisely, for even double-well potentials $\Phi$ we find $p\bpair{t}{\xi\at{t}}=0$ according to the classical Maxwell construction.
\par
The sharp-interface limits of both the viscous regularization \eqref{PDE} and the lattice approximation \eqref{lattice}, however, are more complicated as they involve the flow rule \eqref{lmFR}. In particular, we expect a strong hysteretic behavior of the phase interface since $p\pair{t}{\xi\at{t}}$ is not fixed anymore but can alter in time and affect the propagation mode of the interface. More details concerning the interface hysteresis can be found in \cite{EP04,HH13,Jan23,HJ23a} and the thermodynamical properties of both \eqref{PDE} and \eqref{CahnHilliard} are discussed in \cite{DG17}. 
%
%
\subsection{Numerical simulations and insight}
To elucidate the dynamics of phase interfaces in the case $\nu>0$ we present some numerical results for a typical initial value problem.
%
%
\begin{figure}[ht!]%
\centering{%
\includegraphics[width=0.3\textwidth]{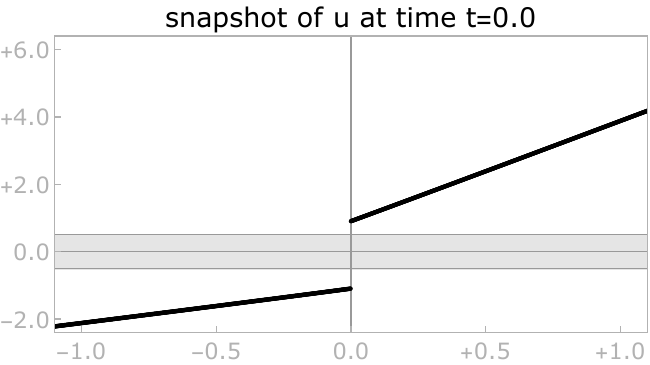}%
\hspace{0.025\textwidth}%
\includegraphics[width=0.3\textwidth]{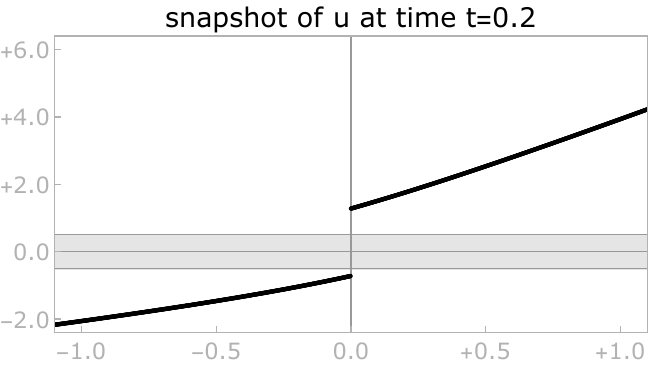}%
\hspace{0.025\textwidth}%
\includegraphics[width=0.3\textwidth]{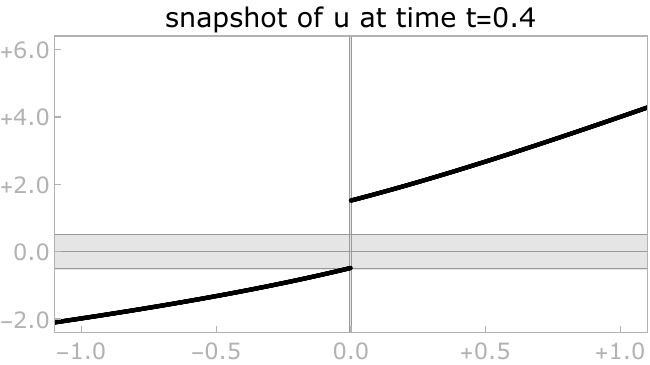}%
\medskip\\%
\includegraphics[width=0.3\textwidth]{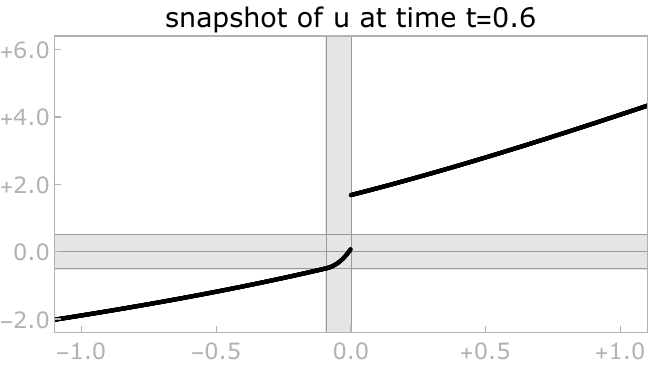}%
\hspace{0.025\textwidth}%
\includegraphics[width=0.3\textwidth]{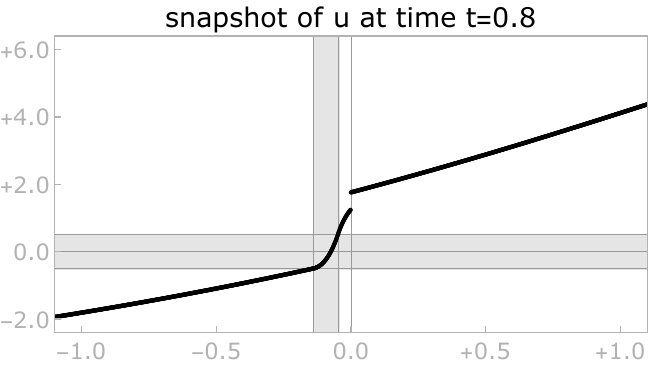}%
\hspace{0.025\textwidth}%
\includegraphics[width=0.3\textwidth]{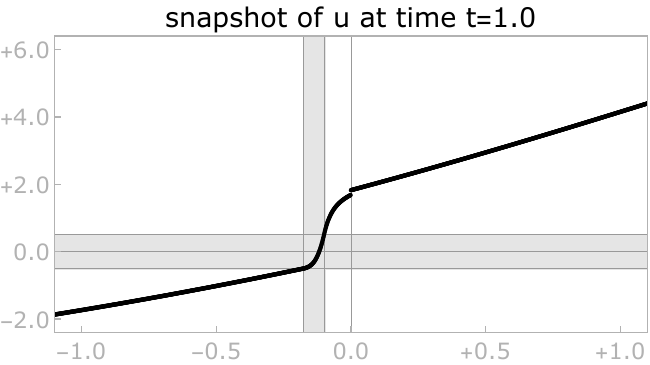}%
\medskip\\%
\includegraphics[width=0.3\textwidth]{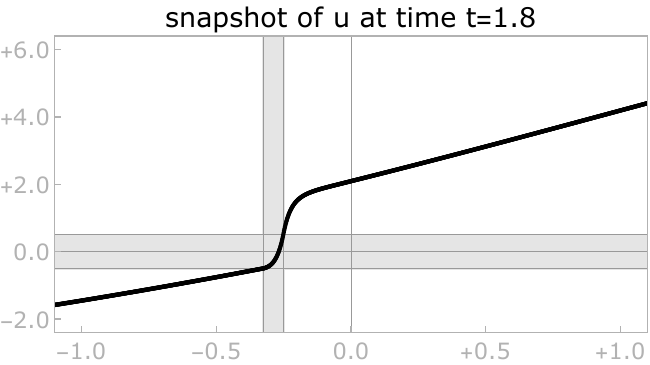}%
\hspace{0.025\textwidth}%
\includegraphics[width=0.3\textwidth]{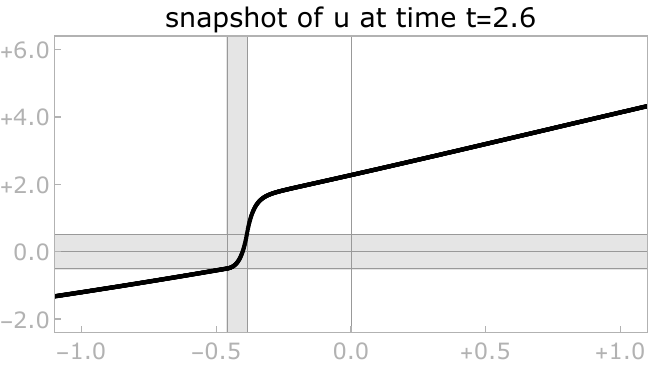}%
\hspace{0.025\textwidth}%
\includegraphics[width=0.3\textwidth]{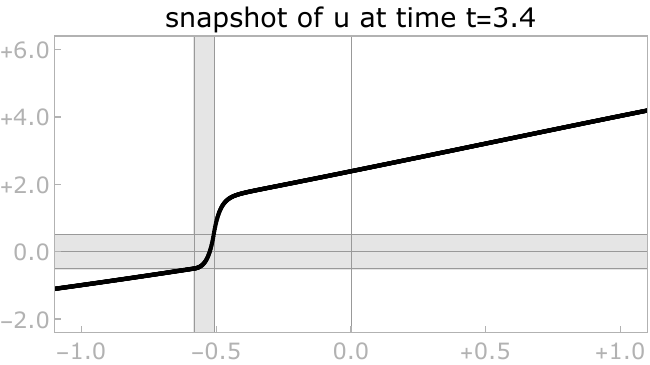}%
}%
\caption{
Typical numerical simulation for the initial data \eqref{IniData} which exhibit a single sharp interface (gray vertical box). For $t>t_\#$, the data near the propagating phase interface can be regarded as an approximate traveling wave, see also the interface curves in Figure \ref{FigNumSim2} as well as the plots of the temporal and spatial derivatives in Figure~\ref{FigNumSim3}. The parameters are $\nu=0.4$, $\ka=0.5$, $\Delta t = 0.01$, $\Delta x = 0.0004$, $\al=-0.62$, $\beta_-=1.0$, $\beta_+=3.0$. The computations were performed with $L=4.0$ but the solution is only shown on the smaller spatial interval $\ccinterval{-1}{+1}$.
}%
\label{FigNumSim1}
\end{figure}
%
%
\paragraph{Discretization} %
We transform \eqref{PDE} into 
\begin{align*}
\partial_t u\pair{t}{x}=\frac{1}{\,\nu^2\,}\,\Bat{w\pair{t}{x}-\Phi^\prime\bat{u\pair{t}{x}}}\,,\qquad 
w\pair{t}{x}-\nu^2\, \partial_x^2 w\pair{t}{x}=\Phi^\prime\bat{u\pair{t}{x}}
\end{align*}
as proposed in \cite{NCP91} and discretize the corresponding initial value problem on the bounded spatial domain $x\in\ccinterval{-L}{+L}$ by means of finite-difference approximations. More precisely, we choose a time step size $0<\Delta t\ll\nu^2$ as well as a spatial grid length $0<\Delta x\ll\nu$, use the formal identification
\begin{align*}
u^n_j=u\pair{n\,\Delta t}{j\,\Delta x}\,,\qquad 
n\in\big\{0\,,\; 1\,,\; \hdots\big\}\,,\qquad 
j\in\big\{-J\,,\;\hdots\,,\;-1\,,\;\,0\,,\;+1\,,\;\hdots\,,\;+J \big\}
\end{align*}
with $J\,\Delta{x}=L$, and compute the discrete data at time step $n+1$ by combining the nonlinear but explicit Euler step
\begin{align*}
u^{n+1}_j=u^n_j+\frac{\,\Delta t\,}{\,\nu^2\,}\, w^{n}_j-\frac{\,\Delta t\,}{\,\nu^2\,}\,\Phi^\prime\at{u^n_j}
\end{align*}
with the implicit but linear auxiliary problem
\begin{align*}
 w^n_j-\at{\frac{\nu}{\,\Delta x\,}}^2\,\bat{w^n_{j+1}+w^n_{j-1}-2\, w^n_j}=\Phi^\prime\bat{u^n_j}\,.
\end{align*}
The latter is complemented by the homogeneous Neumann boundary conditions $w^n_{\pm J \pm 1}=w^n_{\pm J}$ and solved  using a standard LU decomposition of the corresponding elliptic matrix. A similar scheme has been proposed and investigated in \cite{LM12}. 
\par
For simplicity  we restrict our simulations to the trilinear function \eqref{trilin} with fixed $0<\ka<1$ and impose the discrete analogue to the piecewise affine initial data 
\begin{align*}
u_\ini\at{x}= \left\{
\begin{array}{ccc}
-1+\al +\be_-\,x&&\text{for $x<0$}\\
+1+\al+\be_+\, x&&\text{for $x>0$}
\end{array}\right.
\end{align*}
with parameters $\al\in\oointerval{-1+\ka}{+1-\ka}$ and $0<\beta_-<\beta_+$. This choice is consistent with a single standing phase boundary in the sharp-interface model \eqref{lmBD}$+$\eqref{lmSC}$+$\eqref{lmFR} as it guarantees 
\begin{align}
\label{IniData}
u_\ini\at{x}<-\ka\quad \text{for}\quad x<0 \,,\qquad u_\ini\at{x}>+\ka\quad \text{for}\quad x>0\,,\qquad\bjump{u_\ini}_{x=0}=+2
\end{align}
as well as $\xi_-\at{0}=\xi_+\at{0}=0$ .
%
%
\begin{figure}[ht!]%
\centering{%
\includegraphics[width=0.3\textwidth]{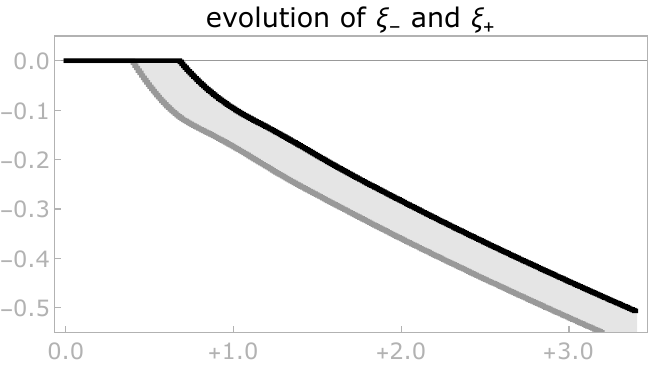}%
}%
\caption{%
The interface boundaries $\xi_-$ (dark gray) and $\xi_+$ (black) as functions of time $t$ for the solution from Figure~\ref{FigNumSim1}.
}%
\label{FigNumSim2}
\end{figure}
%
%
\paragraph{  Interpretation}
Snapshots of the numerical solution and the two interface curves are shown within Figure \ref{FigNumSim1} and \ref{FigNumSim2}, respectively, while Figure \ref{FigNumSim3} displays the corresponding temporal and spatial \mbox{derivatives} \mbox{after} normalization. For small times $t$, the phase interface neither widens nor propagates and the solution attains values outside the spinodal region.  The dynamics are therefore completely determined by the linear diffusion equation $\partial_t p\pair{t}{x}=\partial_x^2p\pair{t}{x}$ since $p\pair{t}{x}=u\pair{t}{x}-\sgn\at{x}$ holds for all $x\in\Rset$. At time $t_*\approx0.4$, however, the data touch the lower critical line $u=-\ka$ and
begin to penetrate the spinodal region from below,  so the width of the phase interface becomes \mbox{positive}. Moreover, the interface starts to propagate to the left (i.e., into the phase $u<-\ka$) since the spatial gradient of $u$ on its right hand side still exceeds the derivative on the left hand side. Between $t_*$ and $t_\#\approx1.0$ we observe a transient regime in which the dynamics produce continuous transitions layers both inside and behind the phase interface while there is no such layer in front of the interface. Finally, for $t\geq t_\#$ we observe a traveling-wave-like behavior near the propagating interface in which the shape of the numerical solution changes only slowly in time and is basically shifted to the left.  A similar behavior can be observed for a wide class of single-interphase initial data, where both the speed and width of the emerging traveling wave strongly depends on $\bjump{\partial _x u_\ini}_{x=0}=\be_+-\be_-$. In the long run, however, the interface decelerates slowly and the width of the interface is getting smaller again. Moreover, for monotone initial data with strong spatial slope changes on either side of the interface one also observes pinning and depinning events at which the propagation mode switches between standing and left- or right-moving.
%
%
\begin{figure}[ht!]%
\centering{%
\includegraphics[width=0.3\textwidth]{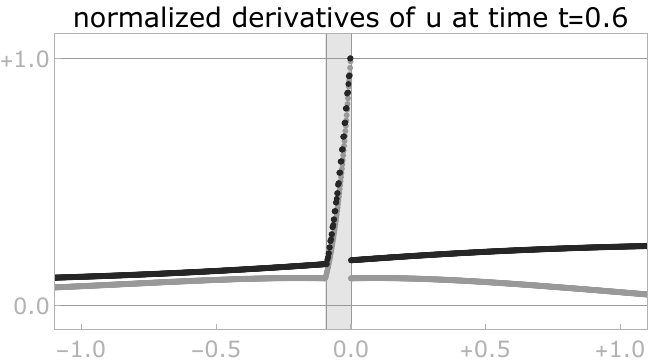}%
\hspace{0.025\textwidth}%
\includegraphics[width=0.3\textwidth]{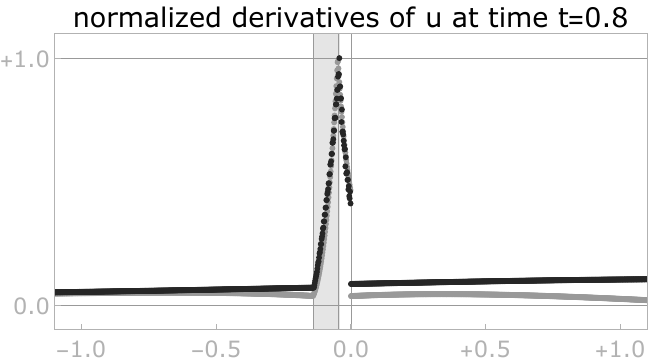}%
\hspace{0.025\textwidth}%
\includegraphics[width=0.3\textwidth]{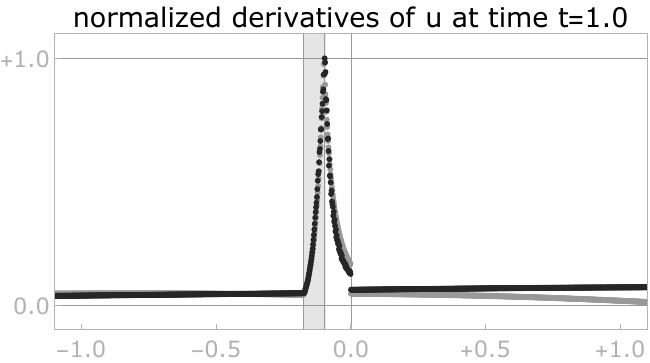}%
\medskip\\%
\includegraphics[width=0.3\textwidth]{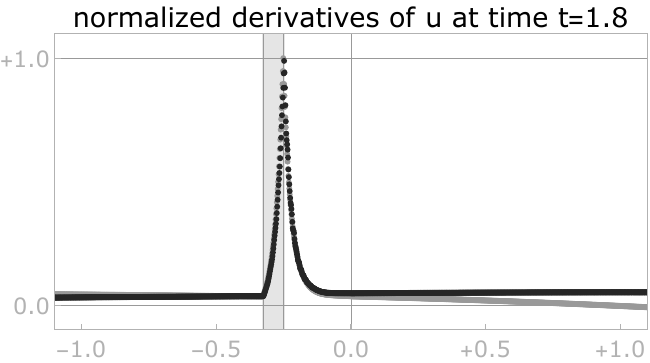}%
\hspace{0.025\textwidth}%
\includegraphics[width=0.3\textwidth]{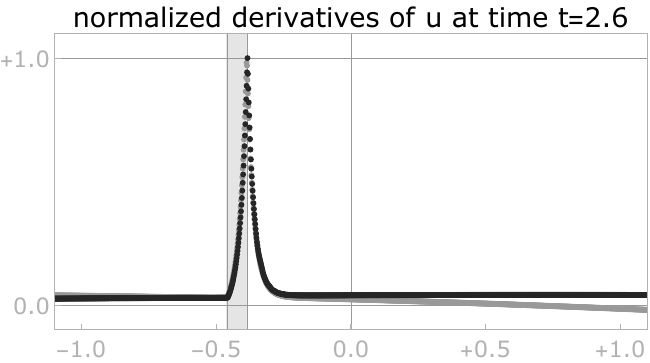}%
\hspace{0.025\textwidth}%
\includegraphics[width=0.3\textwidth]{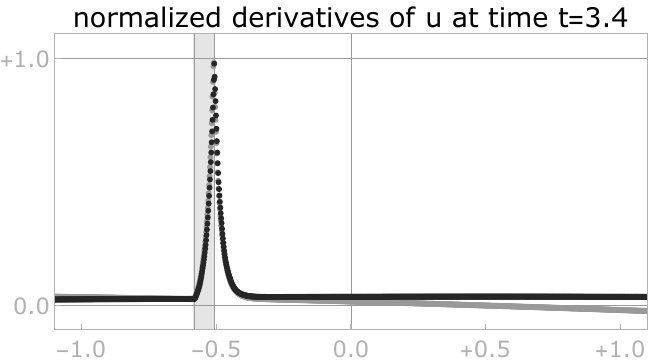}%
}%
\caption{%
Scaled snapshots of $\partial _ t u$ (black) and $\partial _x u$ (gray) for the numerical data from Figure~\ref{FigNumSim1}, where each amplitude has been normalized to $1$. For $t\geq t_\#$, both derivatives differ in the vicinity of the moving phase interface essentially by a scalar factor only.
}%
\label{FigNumSim3}
\end{figure}
%
%
\subsection{Main results and plan of paper}
It seems that traveling wave solutions to the viscous forward-backward diffusion equation \eqref{PDE} have not yet been studied. In this paper we characterize all  waves that are strictly monotone and connect the two phases by a single phase interface. In particular, each wave enters and leaves the spinodal region exactly once. We restrict our considerations to the trilinear function $\Phi^\prime$ from \eqref{trilin} because this allows us to compose the nonlinear wave of the solutions to linear auxiliary ODEs with constant coefficients. Similar traveling waves can be expected to exist for more general bistable nonlinearities but the corresponding analysis is more involved. It requires to characterize the monotonicity of solutions to certain nonlinear differential equations and does not provide explicit formulas.
\par
Our main findings can be summarized as follows.
\begin{result*} Let $0<\ka<1$ and a negative wave speed be fixed.
\begin{enumerate}
\item 
For any given $\nu>0$ there exists a unique two-parameter family of left-moving and strictly increasing traveling wave solutions to \eqref{PDE}$+$\eqref{trilin} but most of these waves grow for small $\nu$ rapidly at $-\infty$.
\item 
There exists a subfamily that converges for $\nu\to0$ to a one-parameter family of traveling wave solutions to the hysteretic limit model \eqref{lmBD}$+$\eqref{lmSC}$+$\eqref{lmFR}. 
\end{enumerate}
Moreover, symmetry arguments imply three similar results for traveling waves that are strictly \mbox{decreasing} and/or move with prescribed positive speed.
\end{result*}
Our first main result is formulated in Theorem \ref{ThmEx1} and sketched in Figure \ref{FigWaves1}. In the proof we first solve linear auxiliary ODEs with constant coefficients for the derivative profile of a traveling wave and evaluate natural matching conditions, see equations \eqref{TWODE2} and \eqref{JC}.
Afterwards we discuss the sign of certain derived quantities in order to single out the values of the remaining free parameters that are compatible with the imposed shape constraint. The subfamily for our second main result is identified in Theorem \ref{ThmEx2} and Theorem \ref{ThmConv1} concerns the passage to the limit  $\nu\to0$, see also Figures \ref{FigWaves2} and \ref{FigWaves4} for an illustration. Finally, Theorem \ref{ThmConv2} provides simplified formulas in the limiting case $\ka\to0$ and the symmetry transformations mentioned above are explained in Figure \ref{FigWaves3}.
%
%
%
\section{Monotone traveling waves}
%
%
In this section we study special solutions to the viscous PDE \eqref{PDE}$+$\eqref{trilin} that penetrate both phases and satisfy the traveling wave ansatz 
\begin{align}
\label{TWAnsatz}
 u\pair{t}{x}=U\at{X}\,,\qquad X=x+S\,t\,.
\end{align}
Here, the parameter $S$ quantifies the negative speed and the wave profile $U$ is supposed to be both continuously differentiable and monotonic. In what follows we assume $S>0$ and characterize all solutions with increasing profile $U$, see Figures \ref{FigWaves1} and \ref{FigWaves2} for an schematic illustration. However, similar arguments can be applied to $S<0$ and/or decreasing functions $U$ and provide traveling waves as depicted in Figure \ref{FigWaves3}.
%
%
\paragraph{Preliminaries and further notations}%
Each traveling wave with increasing $U$ possesses a single phase interface and without loss of generality we can assume that it is centered around $X=0$. We denote the width of this interface by $2\,\Xi$ and introduce the three disjoint intervals
\begin{align}
\label{Intervals}
I_-:=\ocinterval{-\infty}{-\,\Xi}\,,\qquad I_0:=\ccinterval{-\,\Xi}{+\,\Xi}\,,\qquad 
I_+:=\cointerval{+\,\Xi}{+\infty}\,,
\end{align}
which cover $\Rset$. The intervals $I_-$ and $I_+$ represent the regions in which $U$  attains values in either one of the two phases
while it is confined to the spinodal interval inside the interface $I_0$. Notice, however, that $\Xi$ is not known a priori but must be found as part of the solution. Using $\Xi$, the constraints on $U$ can be written as 
\begin{align}
\label{TWConstr1}
 U^\prime\at{X}\geq0\quad\text{for all}\quad X\in I_-\cup I_0\cup I_+\,,\qquad\qquad U\at{X}\in\left\{\begin{array}{lcl}
\ocinterval{-\infty}{-\ka}&&\text{for $X\in I_-$}
\\%
\ccinterval{-\ka}{+\ka}&&\text{for $X\in I_0$}
\\%
\cointerval{+\ka}{+\infty}&&\text{for $X\in I_+$}
\end{array}\right.
\end{align}
and imply  $U\at{\pm\Xi}=\pm\ka$.
%
%
\paragraph{ODE for traveling waves}
Thanks to \eqref{PDE} and \eqref{trilin}, the traveling wave ansatz \eqref{TWAnsatz} yields 
\begin{align}
\label{TWODE1}
S\,\tfrac{\dint}{\dint X} U\at{X}-\nu^2\,S\,\tfrac{\dint{}^3}{\dint  X^3 } U\at{X}
=\tfrac{\dint{}^2}{\dint X^2 }\Bat{ U\at{X}-\sgn_{\,\ka}\bat{U\at{X}}}.
\end{align}
Writing $V=\dint{U}/\dint{X}$, we see that traveling wave solutions to \eqref{PDE} fulfill
\begin{align}
\label{TWODE2}
S\ V\at{X}-\nu^2\,S\,\tfrac{\dint{}^2}{\dint  X^2 } V\at{X}
=\tfrac{\dint{}}{\dint X }\Bat{ V\at{X}-\sgn_{\,\ka}^\prime\bat{U\at{X}}V\at{X}}
\end{align}
thanks to the chain rule. Here, the coefficient function on the right hand side is piecewise constant in view of \eqref{trilin}  and compactly supported with
\begin{align}
\label{DCCoeff}
\sgn_{\,\ka}^\prime\bat{U\at{X}}=\left\{\begin{array}{lcl}
0&&\text{for $X\in I_-$}\,,
\smallskip\\%
\D\frac{1}{\,\ka\,}&&\text{for $X\in I_0$}\,,
\smallskip\\%
0&&\text{for $X\in I_+$}\,.
\end{array}\right.
\end{align}
Moreover, the derivative of the wave profile must satisfy the consistency relation
\begin{align}
\label{TWConstr2}
V\at{X}\geq0\quad\text{for all}\quad X\in I_-\cup I_0\cup I_+\,,\qquad \qquad \int\limits_{-\Xi}^{+\Xi}V\at{X}\dint{X}=2\,\ka\,,
\end{align}
which ensures the existence of a primitive $U$ as in \eqref{TWConstr1}.
%
%
\paragraph{Exponential rates of the determining ODE}
The ODE \eqref{TWODE2} is linear and autonomous on each of the three intervals defined in \eqref{Intervals} and the general solutions involve the exponential rates
\begin{align}
\label{Rates1}
\la_\pm=-\frac{1}{\,2\,\nu^2\,S\,}\pm\sqrt{\at{\frac{1}{\,2\,\nu^2\,S\,}}^2+\frac{1}{\,\nu^2\,}\,}\,,\qquad 
\mu_\pm=+\frac{1-\ka}{2\,\nu^2\,S\,
\ka}\pm\sqrt{\at{\frac{\,1-\ka\,}{\,2\,\nu^2\,S\,\ka\,}}^2+\frac{1}{\,\nu^2\,}\,}\,.
\end{align}
Here, $\la_\pm$ corresponds to both $I_-$ and $I_+$ while $\eta_\pm$ characterize the fundamental solutions inside of the intervals $I_0$. These rates satisfy the order relations
\begin{align}
\label{Rates2}
\la_-<0<\la_+\,,\qquad \mu_-<0<\mu_+
\end{align}
as well as the asymptotic expansions
\begin{align}
\label{Expansion1}
\la_-=-\frac{1}{\,\nu^2\,S\,}-S+O\at{\nu^2}\,,\qquad
\la_+=S+O\at{\nu^2}
\end{align}
and
\begin{align}
\label{Expansion2}
\mu_-=-\frac{\,S\,\ka\,}{\,1-\ka\,}+O\at{\nu^2}\,,\qquad 
\mu_+=\frac{\,1-\ka\,}{\,\nu^2\,S\,\ka\,}+\frac{S\,\ka}{\,1-\ka\,}+O\at{\nu^2}\,.
\end{align}
The subsequent analysis also relies on the following results.
%
%
\begin{lemma}[sign of auxiliary quantities]\label{LemSC}
The estimates
\begin{align}
\label{SC1}
0>\tau_{--}:=-\la_-+\mu_--\frac{1}{\,\nu^2\,S\,\ka\,}
\,,\qquad 
0<\tau_{-+}:=-\la_-+\mu_+-\frac{1}{\,\nu^2\,S\,\ka\,}
\end{align}
and 
\begin{align}
\label{SC2}
0>\tau_{+-}:=-\la_++\mu_--\frac{1}{\,\nu^2\,S\,\ka\,}
\,,\qquad 
0>\tau_{++}:=-\la_++\mu_+-\frac{1}{\,\nu^2\,S\,\ka\,}
\end{align}
hold for all $\nu>0$, $S>0$, $0<\ka<1$. 
\end{lemma}
%
%
\begin{proof}
Setting $\theta:=1/\at{2\,\nu\,S\,\ka}>0$ we compute
\begin{align*}
\nu\,\tau_{--}&=-\,\theta+\sqrt{\ka^2\,\theta^2+1\,}-\sqrt{\at{1-\ka}^2\,\theta^2+1\,}
\,,\\%
\nu\,\tau_{-+}&=-\,\theta+\sqrt{\ka^2\,\theta^2+1\,}+\sqrt{\at{1-\ka}^2\,\theta^2+1\,}
\,,\\ %
\nu\,\tau_{+-}&=-\,\theta-\sqrt{\ka^2\,\theta^2+1\,}-\sqrt{\at{1-\ka}^2\,\theta^2+1\,}
\,,\\%
\nu\,\tau_{++}&=-\,\theta-\sqrt{\ka^2\,\theta^2+1\,}+\sqrt{\at{1-\ka}^2\,\theta^2+1\,}
\,.
\end{align*}
The assertion concerning $\tau_{+-}$ is trivial while the estimate for $\tau_{-+}$  is a consequence of
\begin{align*}
\sqrt{\ka^2\,\theta^2+1\,}>\ka\,\theta\,,\qquad \sqrt{\nat{1-\ka}^2\,\theta^2+1\,}>\at{1-\ka}\,\theta\,.
\end{align*}
The claims on $\tau_{--}$ and $\tau_{++}$ follow from  $\sqrt{\theta^2+1\,}<\theta+1$  due to monotonicity with respect to $\ka$.
\end{proof}
%
%
\paragraph{Matching conditions}
In order to solve the ODE for the traveling waves, see equation \eqref{TWODE2}, we make the ansatz
\begin{align}
\label{TWSolV}
V\at{X}=\left\{\begin{array}{lccr}
A_-\,\exp\at{\la_-\,X}+A_+\,\exp\at{\la_+\,X}&&\text{for $X\in I_-$}\,,
\\%
B_-\,\exp\at{\mu_-\,X}+B_+\,\exp\at{\mu_+\,X}&&\text{for $X\in I_0$}\,,
\\%
C_-\,\exp\at{\la_-\,X}+C_+\,\exp\at{\la_+\,X}&&\text{for $X\in I_+$}\,,
\end{array}\right.
\end{align}
which combines on each interval the respective independent solutions. Due to the discontinuities in the coefficient function \eqref{DCCoeff}, traveling waves cannot be smooth on $\Rset$ and both sides of \eqref{TWODE2} involve Dirac-type terms located at $X=\pm\Xi$. All singular contributions have to balance and using elementary arguments we deduce that $V$ must be continuous everywhere while the jumps of $\nu^2\,S\,\dint{V}/\dint{X}$ and $\sgn_{\ka}^\prime\at{U}\,V$ must equal each other. More specifically, any weak solution to the traveling waves ODE \eqref{TWODE2} complies with the conditions
\begin{align}
\label{JC}
\begin{array}{ccccccc}
V\at{-\Xi-0}&=&V\at{-\Xi+0}\,,\qquad \D\tfrac{\dint }{\dint{X}}V\at{-\Xi-0}&=&\D\tfrac{\dint }{\dint{X}}V\at{-\Xi+0}-
\D\frac{1}{\,\nu^2\,S\,\ka\,}V\at{-\Xi+0}\,,
\medskip\\
V\at{+\Xi+0}&=&V\at{+\Xi-0}\,,\qquad
\D \tfrac{\dint }{\dint{X}}V\at{+\Xi+0}&=&\D\tfrac{\dint }{\dint{X}}V\at{+\Xi-0}-
\frac{1}{\,\nu^2\,S\,\ka\,}V\at{+\Xi-0}\,,
\end{array}
\end{align}
where the notation $\pm0$ indicates one-sided limits. 
%
%
\paragraph{Representation of traveling waves}
Combining \eqref{TWSolV} with \eqref{JC} we arrive at the following equations
\begin{align}
\label{TWCoeff}
\begin{split}
A_-&=\Bat{\tau_{+-}\,B_-\,\exp\at{-\mu_-\,\Xi}+\tau_{++}\,B_+\,\exp\at{-\mu_+\,\Xi}}\,\frac{\,\exp\at{+\la_-\,\Xi}\,}{\la_--\la_+}
\,,\\%
A_+&=\Bat{\tau_{--}\,B_-\,\exp\at{-\mu_-\,\Xi}+\tau_{-+}\,B_+\,\exp\at{-\mu_+\,\Xi}}\,\frac{\,\exp\at{+\la_+\,\Xi}\,}{\la_+-\la_-}
\,,\\%
C_-&=\Bat{\tau_{+-}\,B_-\,\exp\at{+\mu_-\,\Xi}+\tau_{++}\,B_+\,\exp\at{+\mu_+\,\Xi}}\,\frac{\,\exp\at{-\la_-\,\Xi}\,}{\la_--\la_+}
\,,\\%
C_+&=\Bat{\tau_{--}\,B_-\,\exp\at{+\mu_-\,\Xi}+\tau_{-+}\,B_+\,\exp\at{+\mu_+\,\Xi}}\,\frac{\,\exp\at{-\la_+\,\Xi}\,}{\la_+-\la_-}
\,,%
\end{split}
\end{align}
which couple the coefficients in \eqref{TWSolV} and involve the auxiliary quantities from Lemma \ref{LemSC}. In these formulas, the width $\Xi>0$ as well as the real amplitudes $B_-$ and $B_+$  are still independent of the parameters $\nu,\,\ka,\,S$ but the constraints in \eqref{TWConstr2} give rise to certain restrictions.  Leveraging the previous computations, we can make a first statement on solutions to the traveling wave ODE \eqref{TWODE2} for $V=\tfrac{\dint{}}{\dint{X}}U$, combined with the compatibility conditions \eqref{TWConstr2}.
%
\begin{figure}[ht!]%
\centering{%
\includegraphics[width=0.8\textwidth]{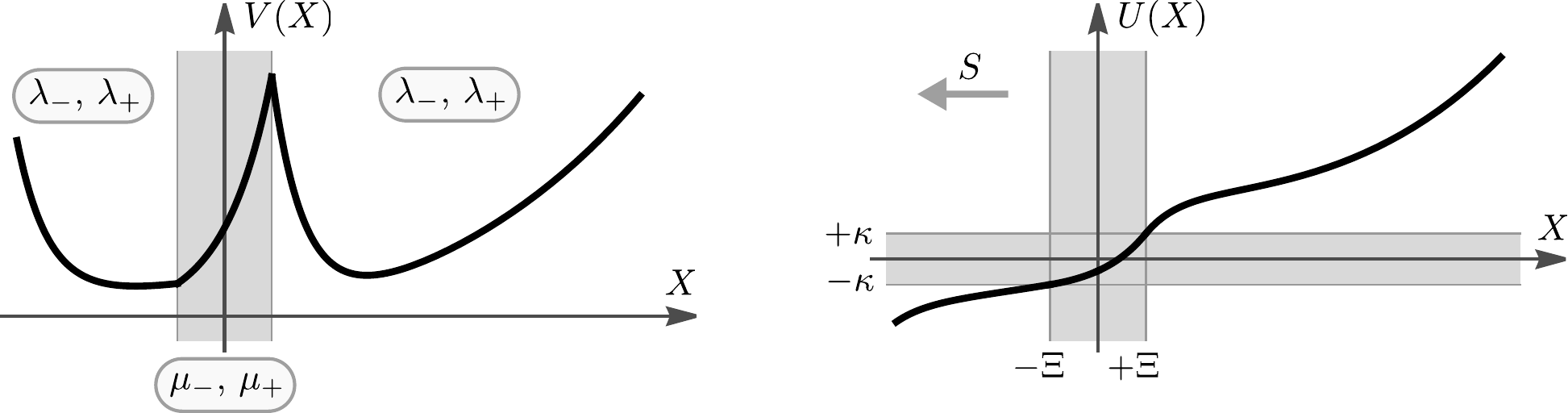}%
}%
\caption{%
The traveling wave profiles $V=\tfrac{\dint{}}{\dint{X}}U$ (\emph{left panel}) and $U$ (\emph{right panel}) from Theorem~\ref{ThmEx1} involve the exponential rates from \eqref{Rates1} and move with speed $-S$. The horizontal and vertical gray stripes indicate the spinodal region and the phase interface, respectively. 
}%
\label{FigWaves1}%
\end{figure}%
%
%
%
\begin{theorem}[family of increasing and left-moving traveling waves] 
\label{ThmEx1}
Let the parameters $\nu>0$, $0<\ka<1$, $S>0$, and $\Xi>0$ be arbitrarily fixed. Then, for any choice of $\si$ with
\begin{align*}
-\frac{\,\tau_{++}\,}{\tau_{+-}}\,\exp\Bat{-\bat{\mu_+-\mu_-}\,\Xi}\;\leq\; \si\;\leq\;
-\frac{\,\tau_{-+}\,}{\tau_{--}}\,\exp\Bat{+\bat{\mu_+-\mu_-}\,\Xi}
\end{align*}
there exists a unique solution to \eqref{TWODE2}$+$\eqref{TWConstr2} with 
\begin{align}
\label{ForBM}
B_-=\si\,B_+\,,\qquad B_+>0
\end{align}
which is strictly positive and also satisfies
\begin{align}
\label{NessSigns}
A_-\geq0\,,\qquad C_+\geq0 \,.
\end{align}
Moreover, other solutions do not exist.
\end{theorem}
%
%
\begin{proof} 
\emph{\ul{Nonexistence of other waves}}\,: 
Let $V$ be any solution to \eqref{TWODE2} that complies with \eqref{TWConstr2}. Then,  \eqref{TWSolV}$+$\eqref{TWCoeff} are satisfied and the coefficients $B_-$, $B_+$ cannot vanish simultaneously due to the integral constraint in \eqref{TWConstr2}. Moreover, \eqref{NessSigns} must also hold true because otherwise $V$ would attain negative values for $X\to-\infty$ or $X\to+\infty$  due to \eqref{Rates2}. Finally, we combine \eqref{TWCoeff} with \eqref{SC1}$+$\eqref{SC2} and discussing the possible sign combinations for $B_-$,  $B_+$ separately, we show that  \eqref{NessSigns} implies the positivity of $B_+$ as well as the bounds for $\si$. 
\par\emph{\ul{Existence for $\si\leq0$}}\,: 
We fix an admissible $\si$, assume \eqref{ForBM} and let \eqref{TWSolV}$+$\eqref{TWCoeff} be satisfied so that $V$ solves \eqref{TWODE2}. We first show that $V$ is positive for any choice of $B_+>0$ and identify the correct value afterwards by a simple scaling argument. From \eqref{Rates2}$+$\eqref{SC1}$+$\eqref{TWCoeff} we deduce
\begin{align}
\label{Sign2}
A_+>0\,,\qquad C_+>0
\end{align} 
and since the lower bound for $\sigma$ also guarantees $A_-\geq0$, the function $V$ is positive on $I_-$ thanks to \eqref{TWSolV}. For $X\in I_0$, the estimate $V\at{X}>0$ follows by $V\at{-\Xi}>0$ and the monotonicity relation
\begin{align*}
\tfrac{\dint }{\dint{X}}V\at{X}=\mu_-\,B_-\,\exp\at{\mu_-\,X}+\mu_+\,B_+\,\exp\at{\mu_+\,X}>0\qquad\text{for}\quad X\in I_0\,.
\end{align*}
On the interval $I_+$, the positivity of $V$ can be inferred directly from \eqref{TWSolV}$+$\eqref{Sign2} in case of $C_-\geq0$ and otherwise via
\begin{align*}
\tfrac{\dint }{\dint{X}}V\at{X}=\la_-\,C_-\,\exp\at{\la_-\,X}+\la_+\,C_+\,\exp\at{\la_+\,X}>0
\qquad\text{for}\quad X\in I_+
\end{align*}
combined with $V\at{+\Xi}>0$. Finally, \eqref{ForBM} transforms the integral condition in \eqref{TWConstr2} into
\begin{align}
\label{ForBP}
B_+=\frac{2\ka}{\;\D\int\limits_{-\Xi}^{+\Xi}\bat{\si\,\exp\at{\mu_-\,X}+\exp\at{\mu_+\,X}}\dint{X}\;}=\frac{\,\ka\,\mu_-\,\mu_+\,}{\;\si\, \mu_+\sinh\at{\mu_-\,\Xi}+\mu_-\sinh\at{\mu_+\,\Xi}\;}
\end{align}
and ensures the existence and uniqueness of $B_+$ and $V$ in dependence of $\si\leq0$. Notice that $B_+$ is truly positive in \eqref{ForBP} since our results above imply $\si\,\exp\at{\mu_-\,X}+\exp\at{\mu_+\,X}>0$ for all $X\in I_0$ (positivity of $V$ for $B_+=1$).
\par\emph{\ul{Existence for $\sigma\geq 0$}}\,:  
We argue similarly to the previous case. By \eqref{Rates2}$+$\eqref{SC2}$+$\eqref{TWCoeff} we have
\begin{align}
\label{Sign1}
A_->0\,,\qquad C_->0
\end{align} 
and the upper bound for $\si$ implies via $C_+\geq0$ that $V$ is positive on $I_+$. The signs of $B_\pm$ imply $V\at{X}>0$ for all $X\in I_0$ and the positivity of $V$ on the interval $I_0$ follows immediately from \eqref{TWSolV}$+$\eqref{Sign1} in the case of $A_+\geq0$ and otherwise by combining $V\at{-\Xi}>0$ with 
\begin{align*}
\tfrac{\dint }{\dint{X}}V\at{X}=\la_-\,A_-\,\exp\at{\la_-\,X}+\la_+\,A_+\,\exp\at{\la_+\,X}<0\qquad\text{for}\quad X\in I_-\,.
\end{align*}
Finally, we employ again \eqref{ForBP}.
\end{proof}
%
%
\paragraph{Isolating a subfamily of relevant traveling waves }
In view of the expansion \eqref{Expansion1} we deduce that traveling waves with $A_->0$ change rapidly near $X\approx-\infty$ and usually do not possess reasonable limits for $\nu\to0$. For $A_-=0$, however, the exponential behavior is much nicer for both $X\to-\infty$ and $X\to+\infty$, and depends asymptotically only on the speed parameter $S$. We thus formulate a corresponding existence and uniqueness result in which $V\at{-\Xi}$ plays the role of an independent parameter.
%
%
\begin{theorem}[family of relevant traveling waves] 
\label{ThmEx2}
For any choice of
$\nu>0$, $0<\ka<1$, $S>0$ and $\eta>0$ there exist unique values $\Xi>0$ and $\si<0$ such that the traveling wave from Theorem~\ref{ThmEx1} satisfies
\begin{align*}
V\at{X}=\eta\,\exp\bat{\la_+\,\at{X+\Xi}}\qquad\text{for all}\quad X\leq-\Xi
\end{align*}
as well as $A_-=0<A_+$, $B_-<0<B_+$ and $C_-,\,C_+>0$.
\end{theorem}
%
%
\begin{proof} 
Using the abbreviations
\begin{align}
\label{DefD}
D_-:=B_-\exp\at{-\mu_-\,\Xi}\,,\qquad
D_+:=B_+\exp\at{-\mu_+\,\Xi}\,,
\end{align}
the conditions $V\at{-\Xi}=\eta$ and $A_-=0$ transform into the linear equations
\begin{align}
\notag
D_-+D_+=\eta\,,\qquad \tau_{+-}\,D_-+\tau_{++}\,D_+=0 \,.
\end{align}
The only solution is given by
\begin{align}
\label{SolD}
D_-=\frac{\tau_{++}}{\,\tau_{++}-\tau_{+-}\,}\,\eta<0\,,\qquad 
D_+=\frac{\tau_{+-}}{\,\tau_{+-}-\tau_{++}\,}\,\eta>0\,,
\end{align}
where the signs follow from  $\tau_{++}-\tau_{+-}=\mu_+-\mu_->0$ and Lemma \ref{LemSC}. Moreover, the integral condition in \eqref{TWConstr2} can be written as
\begin{align}
\label{EqnXi}
2\,\ka=g\bat{\Xi}:=D_-\int\limits_{0}^{2\,\Xi}\exp\bat{\mu_-\,Y}\dint{Y}+D_+\int\limits_{0}^{2\,\Xi}\exp\bat{\mu_+\,Y}\dint{Y}\,.
\end{align}
By direct computations we verify
\begin{align*}
g\at{0}=0\,,\qquad g^\prime\at{0}=2\,\eta>0
\end{align*}
as well as
\begin{align*}
g^{\prime\prime}\bat{\Xi}=4\,D_-\,\mu_-\exp\bat{2\,\mu_-\,\Xi}+4\,D_+\,\mu_+\exp\bat{2\,\mu_+\,\Xi}>0\,,
\end{align*}
where we used the signs of $\mu_\pm$ and $D_\pm$ in \eqref{Rates2} and \eqref{SolD}, respectively. Since $g$ is strictly increasing with $\lim_{\Xi\to\infty}g\at\Xi=\infty$, the existence and uniqueness of a  solution $\Xi>0$ follows immediately. Knowing $D_-$, $D_+$, and $\Xi$ we compute 
\begin{align*}
\si=\frac{\,B_-\,}{\,B_+\,}=\frac{\,D_-\,}{D_+}\,\exp\bat{-\at{\mu_+-\mu_-}\,\Xi}=-\frac{\,\tau_{++}\,}{\tau_{+-}}\,\exp\bat{-\at{\mu_+-\mu_-}\,\Xi}
\end{align*}
and this implies $C_->0$ via
\begin{align*}
-\frac{\,B_-\,}{\,B_+\,}<
 \frac{\tau_{++}}{\,\tau_{+-}\,}\,\exp\bat{+\Xi\,\at{\mu_+-\mu_-}}
\end{align*}
thanks to \eqref{Rates2}$+$\eqref{SC2}$+$\eqref{TWCoeff}.
\end{proof}
Similar ODE techniques have been used in \cite{HN23} to compute traveling phase interfaces in a particle model with nonlocal mean-field interaction and inhomogeneities.
%
%
\begin{figure}[ht!]%
\centering{%
\includegraphics[width=0.8\textwidth]{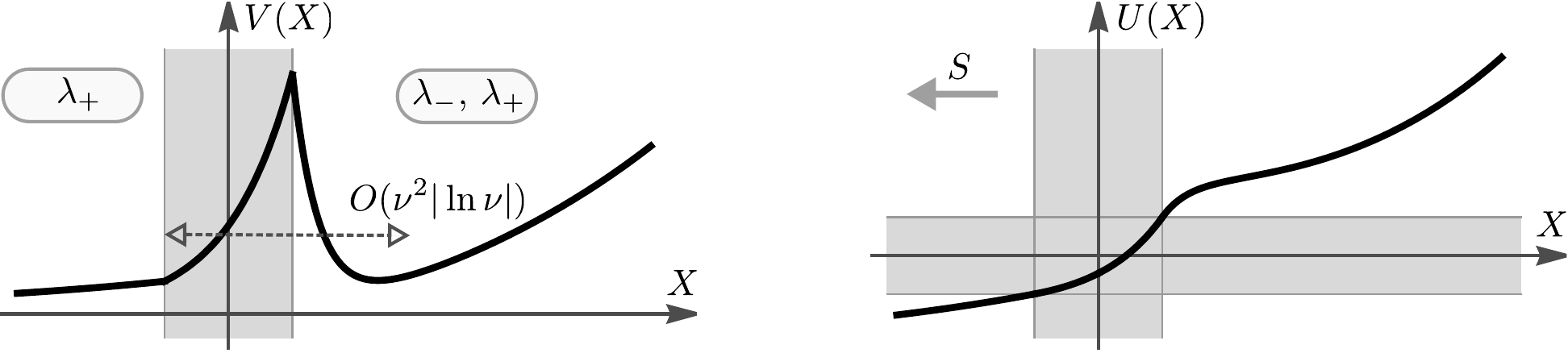}%
}%
\caption{%
Cartoon of the traveling wave from Theorem~\ref{ThmEx2}. The exponential mode corresponding to the rate $\la_-$ appears only in the back of the wave and describes a transition layer that produces a jump of $U$ in the limit of vanishing viscosity. The other waves from Theorem~\ref{ThmEx1} do not converge as $\nu\to0$. 
}%
\label{FigWaves2}%
\end{figure}%
%
%
\begin{figure}[ht!]%
\centering{%
\includegraphics[width=0.99\textwidth]{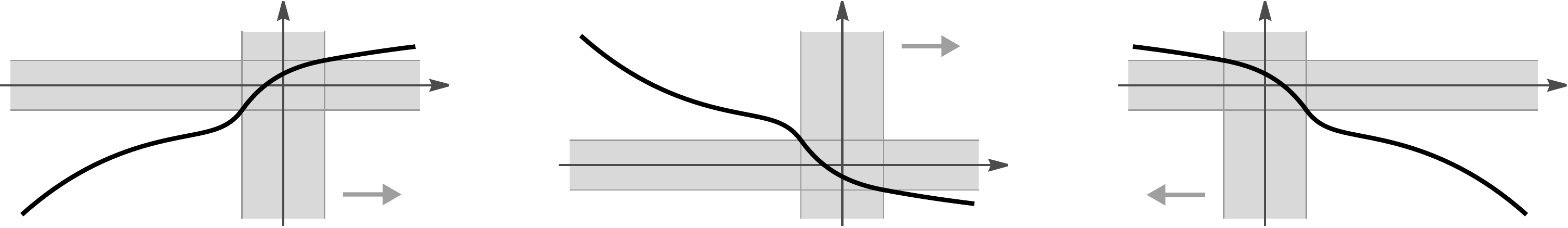}%
}%
\caption{%
Traveling waves also exists with $S<0$ and/or decreasing $U$ and can be related to reflections on the vertical and/or horizontal axis. Here illustrated for the wave from Figure \ref{FigWaves2}.
}%
\label{FigWaves3}%
\end{figure}%
%
%
\section{Convergence of traveling waves}
%
%
In this section we study the limits $\nu\to0$ and $\ka\to0$ and provide simplified formulas for the corresponding traveling waves as illustrated in Figure \ref{FigWaves4}.
%
%
\begin{theorem}[limit of vanishing viscosity]
\label{ThmConv1}
The waves from Theorem~\ref{ThmEx2} satisfy
\begin{align}
\label{AsyXi}
\Xi=\frac{\,\nu^2\,S\,\ka\,}{\,2\,\at{1-\ka}\,}\,\ln\at{\frac{\,2\,\at{1-\ka}^2\,}{\,\nu^2\,S\,\eta\,}}\,\Bat{1+O\at{\nu^2}}\quad\xrightarrow{\;\;\nu\to0\;\;}\quad 0
\end{align}
as well as the pointwise convergence
\begin{align}
\label{ConvU}
U\at{X}\quad \xrightarrow{\;\;\nu\to0\;\;}\quad \left\{
\begin{array}{rc}
-\ka-\eta\,S^{-1}\,\Bat{1-\exp\at{S\,X}}&\text{for $X<0$}\\
2-\ka+\bat{\eta+2\,S}\,S^{-1}\,\Bat{\exp\at{S\,X}-1}&\text{for $X>0$}\\
\end{array}
\right.
\end{align}
provided that the parameters $0<\ka<1$, $S>0$, and $\eta>0$ are chosen independent of $\nu$.
\end{theorem}
%
%
\begin{proof}
\emph{\ul{Preliminaries}}\,: Formula \eqref{Expansion1} ensures
\begin{align}
\label{AsyAux0}
\la_+-\la_-=\frac{1}{\,\nu^2\,S\,}\,\Bat{1+O\at{\nu^2}}
\end{align}
and for the auxiliary quantities from Lemma \ref{LemSC} we obtain
\begin{align}
\label{AsyAux1}
\tau_{--}=-\frac{1-\ka}{\,\nu^2\,S\,\ka\,}\,\Bat{1+O\at{\nu^2}}\,,\qquad
\tau_{-+}=\frac{S}{\,1-\ka\,}\,\Bat{1+O\at{\nu^2}}
\end{align}
as well as
\begin{align}
\label{AsyAux2}
\tau_{+-}=-\frac{1}{\,\nu^2\,S\,\ka\,}\,\Bat{1+O\at{\nu^2}}\,,\qquad
\tau_{++}=-\frac{1}{\,\nu^2\,S\,}\,\Bat{1+O\at{\nu^2}}
\end{align}
by straight forward Taylor expansions. Inserting this into \eqref{SolD} gives
\begin{align}
\label{AsyD}
D_-=-\frac{\ka\,\eta}{\,1-\ka\,}\,\Bat{1+O\at{\nu^2}}\,,\qquad D_+=+\frac{\eta}{\,1-\ka\,}\,\Bat{1+O\at{\nu^2}}
\end{align}
for the modified coefficients from \eqref{DefD}\,.
\par\emph{\ul{Asymptotics of interface width}}\,:
The equation for $\Xi$ in \eqref{EqnXi} can be written as 
\begin{align*}
\exp\bat{2\,\mu_+\,\Xi}=1+\frac{\,2\,\ka\,\mu_+\,}{D_+}-\frac{D_-\,\mu_+}{D_+\,\mu_-}\,\Bat{\exp\at{2\,\mu_-\,\Xi}-1}\,,
\end{align*}
where $\mu_-$ and $\mu_+$ exhibit different asymptotic properties according to \eqref{Expansion2}. Using this and \eqref{AsyD} we conclude that $\mu_+\,\Xi$ is large but still of order $\abs{\ln\nu}$ while $\Xi$ is small and of order $\nu^2\abs{\ln\nu}$. In particular, we get
\begin{align}
\label{AsyExp1}
\exp\bat{2\,\mu_-\,\Xi}=1+O\bat{\nu^2\nabs{\ln\nu}}\,,\qquad
\exp\bat{\la_+\,\Xi}=1+O\bat{\nu^2\nabs{\ln\nu}}
\end{align}
as well as
\begin{align}
\label{AsyExp2}
\exp\bat{2\,\mu_+\,\Xi}=\frac{\,2\,\at{1-\ka}^2\,}{\nu^2\,S\,\eta}\Bat{1+O\at{\nu^2\nabs{\ln\nu}}}\,.
\end{align}
The asymptotic solution formula \eqref{AsyXi} now follows by 
taking the logarithm and inserting \eqref{Expansion2}.
\par\emph{\ul{Limit of  $\,V\!$ on $I_-$ and $I_0$}}\,:
The asymptotics of $\la_+$ and $\Xi$ in \eqref{Expansion1} and \eqref{AsyXi}, respectively, guarantee the pointwise convergence
\begin{align*}
V\at{X}=A_+\,\exp\at{\la_+\,X}\quad \xrightarrow{\;\;\nu\to0\;\;}\quad \eta\,\exp\at{S\,X}
\end{align*}
for any $X<0$ thanks to $A_+=\eta\,\exp\at{+\la_+\,\Xi}$. Moreover,  the integral constraint in  \eqref{TWConstr2} implies that
\begin{align*}
\chi_{I_0}\at{X}\,V\at{X}\quad \xrightarrow{\;\;\nu\to0\;\;}\quad 2\,\ka\,\delta_0\at{X}
\end{align*}
holds in the sense of measures, where $\chi_{I_0}$ denotes the characteristic function of the interval $I_0$ from \eqref{Intervals} and $\delta_0$ is the Dirac distribution centered in $X=0$.
\par\emph{\ul{Limit of $\,V$on $I_+$}}\,:
Combining \eqref{TWCoeff} and \eqref{DefD} with the asymptotic results \eqref{AsyAux0}, \eqref{AsyAux1}, \eqref{AsyD}, \eqref{AsyExp1}, and \eqref{AsyExp2}
provides 
\begin{align*}
C_{+}&=\Bat{\tau_{--}\, D_-\,\exp\at{2\,\mu_-\,\Xi}+\tau_{-+}\, D_+\,\exp\at{2\,\mu_+\,\Xi}}\,\frac{\,\exp\at{-\la_+\,\Xi}\,}{\la_+-\la_-}
\\&=
\at{
\frac{1-\ka}{\,\nu^2\,S\,\ka\,}\,\frac{\,\ka\,\eta\,}{\,1-\ka\,}
+
\frac{S}{\,1-\ka\,}\,\frac{\eta}{\,1-\ka\,}\,\frac{\,2\,\at{1-\ka}^2\,}{\nu^2\,S\,\eta}}\,\nu^2\,S\,\Bat{1+O\at{\nu^2\nabs{\ln\nu}}}
\\&=\bat{\eta+2\,S}\,\Bat{1+O\at{\nu^2\nabs{\ln\nu}}}\,,
\end{align*}
so the pointwise convergence
\begin{align*}
C_+\,\exp\at{\la_+\,X}\quad \xrightarrow{\;\;\nu\to0\;\;}\quad \bat{\eta+2\,S}\exp\at{S\,X}
\end{align*}
holds for any fixed $X>0$. Using similar arguments as well as \eqref{AsyAux2} we derive
\begin{align*}
C_-\,\exp\at{\la_-\,X}\quad \xrightarrow{\;\;\nu\to0\;\;}\quad 0
\end{align*}
due to the singular asymptotics of $\la_-$ in \eqref{Expansion1} but at the same time we have
\begin{align*}
\int\limits_{+\Xi}^{+\infty} C_-\,\exp\at{\la_-\,X}\dint{X}
&=%
-\frac{C_-}{\la_-}\,\exp\at{\la_-\,\Xi}
\\&=%
\Bat{\tau_{+-}\, D_-\,\exp\at{2\,\mu_-\,\Xi}+\tau_{++}\, D_+\,\exp\at{2\,\mu_+\,\Xi}}\,\frac{1}{\,\la_-\,\at{\la_+-\la_-}\,}
\\&=%
\at{\frac{1}{\,\nu^2\,S\,\ka\,}\,\frac{\ka\,\eta}{\,1-\ka\,}-\frac{1}{\,\nu^2\,S\,}\,\frac{\eta}{\,1-\ka\,}\,\frac{\,2\,\at{1-\ka}^2\,}{\nu^2\,S\,\eta}}\,\bat{-\nu^4\,S^2}
\,\Bat{1+O\at{\nu^2\nabs{\ln\nu}}}
\\&=2\,\at{1-\ka}\,\Bat{1+O\at{\nu^2\nabs{\ln\nu}}}\,.
\end{align*}
We thus obtain
\begin{align*}
\chi_{I_+}\at{X}\,C_-\,\exp\at{\la_-\,X}
\quad \xrightarrow{\;\;\nu\to0\;\;}\quad 2\,\at{1-\ka}\,\delta_0\at{X}
\end{align*}
in the sense of measures. 
\par\emph{\ul{Limits of $\,V\!$ and $U$}}\,:
Our results derived so far guarantee
\begin{align*}
V\at{X}\quad \xrightarrow{\;\;\nu\to0\;\;}\quad 2\,\delta_0\at{X}+
\left\{
\begin{array}{rc}
\eta\,\exp\at{S\,X}&\text{for $X<0$}\\
\bat{\eta+2\,S}\,\exp\at{S\,X}&\text{for $X>0$}\\
\end{array}
\right.
\end{align*}
in the sense of measures as well as
\begin{align*}
U\at{X}=-\ka + \int\limits_{X}^{-\Xi} V\at{Y}\dint{Y}
\quad\xrightarrow{\;\;\nu\to0\;\;}\quad
-\ka + \int\limits_{X}^{0} \eta\,\exp\at{S\,X}\dint{Y}=
-\ka+\eta\,S^{-1}\bat{\exp\at{S\,X}-1}
\end{align*}
for any fixed $X<0$. Both results imply \eqref{ConvU} and the proof is complete.
\end{proof}
%
%
\begin{figure}[ht!]%
\centering{%
\includegraphics[width=0.8\textwidth]{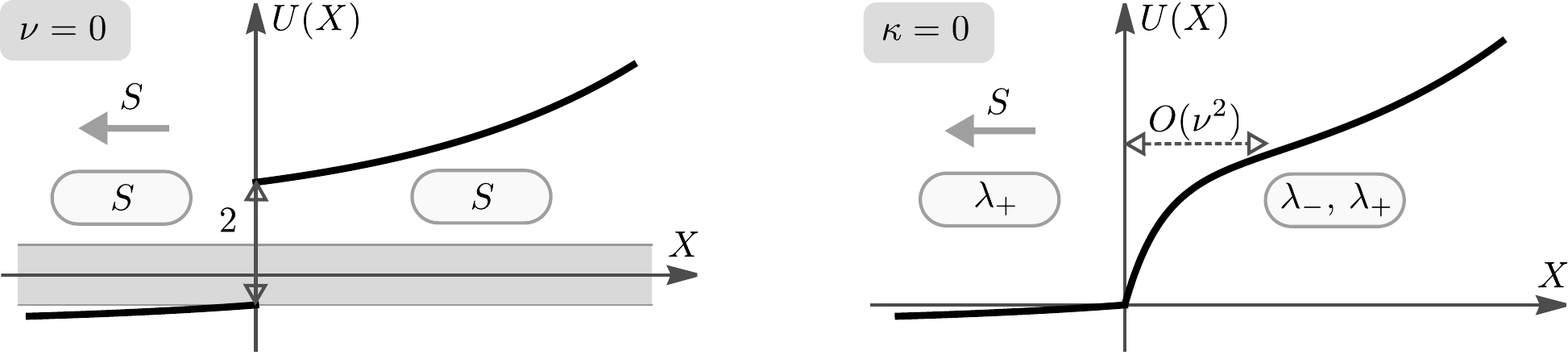}%
}%
\caption{%
Cartoon of the limiting wave profiles $U$ from Theorem \ref{ThmConv1} (\emph{left panel}) and Theorem \ref{ThmConv2} (\emph{right panel}). For the latter, the contributions with rate $\la_-$ describe a smooth transition layer with width of order $\nu^2$ that replaces the jump of height $2$.
}%
\label{FigWaves4}%
\end{figure}%
%
%
\paragraph{Compatibility with the hysteretic limit model} 
We show that the limit wave in Theorem~\ref{ThmConv1} describes a left-moving single-interface solution to the ill-posed diffusion equation that is compatible with the hysteretic limit model from \S\ref{secIntro}. In fact, with 
\begin{align*}
P\at{X}=U\at{X}-\sgn_\ka\bat{U\at{X}}=U\at{X}-\sgn\at{X}
\end{align*}
we find
\begin{align*}
\bjump{U}_{X=0}=2\,,\qquad \bjump{P}_{X=0}=0\,,\qquad
\bjump{\tfrac{\dint}{\dint{X}}P}_{X=0}=2\,S
\end{align*}
as well as
\begin{align*}
S\,\tfrac{\dint}{\dint{X}}U\at{X}=S\,V\at{X}=\tfrac{\dint}{\dint{X}}V\at{X}=\tfrac{\dint^2}{\dint{X^2}}U\at{X}=\tfrac{\dint^2}{\dint{X^2}}P\at{X}\qquad \text{for}\qquad X\neq0
\end{align*}
and conclude that the function $u$ from \eqref{TWAnsatz} satisfies the single-interface condition \eqref{singInt} with $\xi\at{t}=-S\,t$ as well as the bulk diffusion \eqref{lmBD} and the Stefan condition \eqref{lmSC}. Moreover, we have
\begin{align*}
S>0\,,\qquad P\at{0}=+1-\ka
\end{align*}
in accordance with the flow rule \eqref{lmFR}.
\bigpar
As our last result, we investigate the limit as $\kappa \to 0$ in which $\Phi^\prime$ becomes bilinear. In this special case we already find sharp interfaces even for $\nu > 0$ since the spinodal region consists only of the isolated point $U = 0$. The corresponding sharp-interface model for $\nu\to0$, however, is basically the same as for $\ka>0$ and also exhibits strong hysteresis. We further mention that the existence of traveling wave solutions of the viscous PDE \eqref{PDE} with bilinear $\Phi^\prime$ has already been proven directly in \cite{Jan23}, and we recover the same limit formula. 
%
%
\begin{theorem}[bilinear limiting case]
\label{ThmConv2}
Let $\nu>0$, $S>0$, and $\eta>0$ be fixed. Then we have
\begin{align}
\label{X3}
\Xi=\frac{\,\nu^2\,S\,\ka\,}{2}\,\ln\at{1+\frac{2}{\,\nu^2\,S\,\eta\,}}\,\Bat{1+O\at{\ka}}\quad\xrightarrow{\;\;\ka\to0\;\;}\quad 0
\end{align}
and
\begin{align}
\notag
U\at{X}\quad\xrightarrow{\;\;\ka\to0\;\;}\quad \frac{\eta}{\,\la_+\,}\,\Bat{\exp\at{\la_+\,X}-1}\,+\,\frac{2}{\,\sqrt{1+4\,\nu^2\,S^2\,}\,}\left\{\begin{array}{cc}
0&\text{for $X<0$}\\\!\exp\at{\la_+\,X}-\exp\at{\la_-\,X}&\text{for $X\geq 0$}\end{array}\right.
\end{align}
for the traveling waves from Theorem~\ref{ThmEx2}.
\end{theorem}
%
\begin{proof}
\emph{\ul{Elementary asymptotics}}\,: 
From \eqref{Rates1} we deduce that
\begin{align}
\label{X1}
\mu_-=-S\,\ka\,\Bat{1+O\at{\ka}}\,,\qquad \mu_+=\frac{1}{\,\nu^2\,S\,\ka\,}\,\Bat{1+O\at{\ka}}
\end{align}
and by direct computations we verify
\begin{align*}
\tau_{--}=-\frac{1}{\,\nu^2\,S\,\ka\,}\,\Bat{1+O\at{\ka}}\,,\qquad
\tau_{-+}=\la_+\,\Bat{1+O\at{\ka}}\,\qquad
\end{align*}
as well as
\begin{align*}
\tau_{+-}=-\frac{1}{\,\nu^2\,S\,\ka\,}\,\Bat{1+O\at{\ka}}\,,\qquad
\tau_{++}=\la_-\,\Bat{1+O\at{\ka}}\,.
\end{align*}
Moreover, in view of \eqref{DefD} we get
\begin{align}
\label{X2}
D_-=\la_-\,\nu^2\,S\,\ka\,\eta\,\Bat{1+O\at{\ka}}\,,\qquad\, D_+=\eta\,\Bat{1+O\at{\ka}}\,,
\end{align}
so the formulas
\begin{align}
\label{X4}
\tau_{--}\,D_-=-\la_-\,\eta\,\Bat{1+O\at{\ka}}\,,\qquad \tau_{-+}\,D_+=+\la_+\,\eta\,\Bat{1+O\at{\ka}}
\end{align}
and
\begin{align}
\label{X5}
\tau_{+-}\,D_-=-\la_-\,\eta\,\Bat{1+O\at{\ka}}\,,\qquad \tau_{++}\,D_+=+\la_-\,\eta\,\Bat{1+O\at{\ka}}
\end{align}
follow immediately.
\par\emph{\ul{Effective formulas for the interface width}}\,: 
Equation \eqref{EqnXi} can be written as 
\begin{align}
\label{X9}
2=\frac{D_-}{\,\ka\,\mu_-\,}\,\bat{\exp\at{2\,\mu_-\,\Xi}-1}+\frac{D_-}{\,\ka\,\mu_+\,}\,\bat{\exp\at{2\,\mu_+\,\Xi}-1}
\end{align}
and implies 
\begin{align}
\notag
2\,\mu_+\,\Xi\leq 1+\ln\at{\frac{\,2\,\ka\,\mu_+\,}{D_+}}
\end{align}
because both $D_-/\mu_-$ and $D_+/\mu_+$ are positive due to \eqref{Rates2}+\eqref{SolD}.  Since the right hand side is uniformly bounded for $\ka\to0$ according to \eqref{X1} and \eqref{X2}, we infer that $\Xi$ is of order $O\at{\ka}$. Using this and \eqref{X1}$+$\eqref{X2}$+$\eqref{X9} once again we obtain
\begin{align}
\label{X6}
\exp\at{2\,\mu_-\,\Xi}=1+O\at{\ka^2}\,,\qquad 
\exp\at{2\,\mu_+\,\Xi}=1+\frac{2}{\,\nu^2\,S\,\eta\,}+O\at{\ka}
\end{align}
and this implies in combination with the expansion for $\mu_+$ the claim \eqref{X3}.
\par\emph{\ul{Asymptotics of the wave coefficients}}\,: 
The expansions \eqref{X4}$+$\eqref{X5} and \eqref{X6} provide
\begin{align*}
C_-&=\Bat{\bat{\tau_{+-}\, D_-+\tau_{++}\,D_+}+\tau_{++}\,D_+\bat{\exp\nat{2\,\mu_+\,\Xi}-1}}\,\frac{1}{\,\la_--\la_+\,}\,\Bat{1+O\at{\ka}}
\\&=\Bat{0+\la_-\,\eta\,\frac{2}{\,\nu^2\,S\,\eta\,}}\,\frac{1}{\,\la_--\la_+\,}
\,\Bat{1+O\at{\ka}}\\&=-\frac{\la_-}{\,\la_+-\la_-\,}
\,\frac{2}{\,\nu^2\,S\,}+O\at{\ka}
\end{align*}
and
\begin{align*}
C_+&=\Bat{\bat{\tau_{--}\, D_-+\tau_{-+}\,D_+}+\tau_{-+}\,D_+\bat{\exp\nat{2\,\mu_+\,\Xi}-1}}\,\frac{1}{\,\la_+-\la_-\,}\,\Bat{1+O\at{\ka}}
\\&=\Bat{\at{\la_+-\la_-}\,\eta+\la_+\,\eta\,\frac{2}{\,\nu^2\,S\,\eta\,}}\,\frac{1}{\,\la_+-\la_-\,}
\,\Bat{1+O\at{\ka}}
\\&=\eta+\frac{2}{\,\nu^2\,S\,}\,\frac{\la_+}{\,\la_+-\la_-\,}
\,+O\at{\ka}\,.
\end{align*}
Combining this with $A_+=\eta\,\exp\at{+\la_+\,\Xi}=\eta+O\at{\ka}$ we get
\begin{align*}
V\at{X}\quad\xrightarrow{\;\;\ka\to0\;\;}\quad\eta\,\exp\at{\la_+\,X}\,+\,\frac{2}{\,\nu^2\,S\,\at{\la_+-\la_-}\,}\left\{\begin{array}{cc}
0&\text{for $X<0$}\\\la_+\exp\at{\la_+\,X}-\la_-\exp\at{\la_-\,X}&\text{for $X\geq 0$}\end{array}\right.
\end{align*}
and the second claim follows since we also have
\begin{align*}
U\at{X}=-\ka + \int\limits_{X}^{-\Xi} V\at{Y}\dint{Y}
\quad\xrightarrow{\;\;\ka\to0\;\;}\quad
 \int\limits_{X}^{0} \eta\,\exp\at{\la_+\,X}\dint{Y}=
\eta\,\la_+^{-1}\,\Bat{\exp\at{\la_+\,X}-1}
\end{align*}
for any fixed $X<0$.
\end{proof}
The limit formula in Theorem \ref{ThmConv2} covers all distributional solutions to the linear but inhomogeneous ODE
\begin{align*}
S\tfrac{\dint}{\dint{X}}U\at{X}-\nu^2\,S\tfrac{\dint^3}{\dint{X^3}}U\at{X}=\tfrac{\dint^2}{\dint{X^2}}\Bat{U\at{X}-\sgn\bat{X}}
\end{align*}
and has already been given in \cite{Jan23}. 
\appendix
%
\subsection*{Acknowledgments}
%
\addcontentsline{toc}{section}{Acknowledgments}
This work has been supported by the German Research Foundation (DFG) by the individual grant HE 6853/3-1.
%
%
\subsection*{List of symbols}
%
\addcontentsline{toc}{section}{List of symbols}
\begin{flushleft}
\small
\begin{tabular}{clclcl}
&$\nu$ and $\ka$&&viscosity and parameter for the bistable nonlinearity&&
\eqref{PDE}, \eqref{trilin}
\\%
&$U$ and $V$&&traveling wave profile and its derivative&&\eqref{TWODE1}, \eqref{TWODE2}
\\%
&$X$ and $S$&&coordinate in the coming frame and negative wave speed &&\eqref{TWAnsatz}
\\%
&$\Xi$&& interphase parameter for traveling waves&&\eqref{TWConstr1}
\\%
&$I_-,\,I_0,\,I_+$&&different intervals representing the phases and the interface &&\eqref{Intervals}
\\%
&$\la_\pm$ and $\mu_\pm$&&exponential rates depending on $\nu$, $\ka$, and $S$&&\eqref{Rates1}
\\%
&$\tau_{-\pm}$ and $\tau_{+\pm}$ &&auxiliary quantities related to $\la_\pm$ and $\mu_\pm$&& Lemma \ref{LemSC}
\\%
&$A_\pm,\,B_\pm,\,C_\pm$&&ODE coefficients for $V$ in the different intervals&&\eqref{TWCoeff}
\\%
&$D_\pm$&&modified coefficients inside the spinodal region&&\eqref{DefD}
\\%
&$\eta$ and $\sigma$&&additional parameters for traveling waves
&&Theorems \ref{ThmEx1} and \ref{ThmEx2}
\\%
&$t$, $x$, $\xi_\pm$, $u$, $p$&&variables in the underlying PDE&&Section \ref{secIntro}
\end{tabular}
\end{flushleft}
%
%
%
%
%
\addcontentsline{toc}{section}{References}
\bibliographystyle{alpha}
\bibliography{paper}

\begin{thebibliography}{BBMN12}

\bibitem[BBMN12]{BBMN11}
G.~Bellettini, L.~Bertini, M.~Mariani, and M.~Novaga.
\newblock Convergence of the one-dimensional {C}ahn-{H}illiard equation.
\newblock {\em SIAM J. Math. Anal.}, 44(5):3458--3480, 2012.

\bibitem[BFG06]{BFG06}
G.~Bellettini, G.~Fusco, and N.~Guglielmi.
\newblock A concept of solution and numerical experiments for forward-backward
  diffusion equations.
\newblock {\em Discrete Contin. Dyn. Syst.}, 16(4):783--842, 2006.

\bibitem[BGN13]{BGN13}
G.~Bellettini, C.~Geldhauser, and M.~Novaga.
\newblock Convergence of a \mbox{semidiscrete} scheme for a forward-backward
  parabolic equation.
\newblock {\em Adv. \mbox{Differential} Equations}, 18(5-6):495--522, 2013.

\bibitem[DG17]{DG17}
W.~Dreyer and C.~Guhlke.
\newblock Sharp limit of the viscous {C}ahn-{H}illiard equation and
  thermodynamic consistency.
\newblock {\em Contin. Mech. Thermodyn.}, 29(4):913--934, 2017.

\bibitem[Ell85]{Ell85}
C.~M. Elliott.
\newblock The {S}tefan problem with a nonmonotone constitutive relation.
\newblock {\em IMA J. Appl. Math.}, 35(2):257--264, 1985.
\newblock Special issue: IMA conference on crystal growth (Oxford, 1985).

\bibitem[EP04]{EP04}
L.~C. Evans and M.~Portilheiro.
\newblock Irreversibility and hysteresis for a forward-backward diffusion
  equation.
\newblock {\em Math. Models Methods Appl. Sci.}, 14(11):1599--1620, 2004.

\bibitem[GN11]{GN11}
C.~Geldhauser and M.~Novaga.
\newblock A semidiscrete scheme for a one-dimensional {C}ahn-{H}illiard
  equation.
\newblock {\em Interfaces Free Bound.}, 13(3):327--339, 2011.

\bibitem[GT10]{GT10}
B.~H. Gilding and A.~Tesei.
\newblock The {R}iemann problem for a forward-backward parabolic equation.
\newblock {\em Phys. D}, 239(6):291--311, 2010.

\bibitem[HH13]{HH13}
M.~Helmers and M.~Herrmann.
\newblock Interface dynamics in discrete forward-backward diffusion equations.
\newblock {\em SIAM Multiscale Model. Simul.}, 11(4):1261--1297, 2013.

\bibitem[HH18]{HH18}
M.~Helmers and M.~Herrmann.
\newblock Hysteresis and phase transitions in a lattice regularization of an
  ill-posed forward-backward diffusion equation.
\newblock {\em Arch. Ration. Mech. Anal.}, 230(1):231--275, 2018.

\bibitem[Hil89]{Hil89}
M.~Hilpert.
\newblock On uniqueness for evolution problems with hysteresis.
\newblock In {\em \mbox{Mathematical} models for phase change problems
  (\'{O}bidos, 1988)}, volume~88 of {\em Internat. Ser. Numer. Math.}, pages
  377--388. Birkh\"{a}user, Basel, 1989.

\bibitem[HJ23a]{HJ23a}
M.~Herrmann and D.~Jan{\ss}en.
\newblock Hysteretic dynamics of phase interfaces in bilinear forward-backward
  diffusion equations.
\newblock in preparation, 2023.

\bibitem[HJ23b]{HJ23b}
M.~Herrmann and D.~Jan{\ss}en.
\newblock Phase interfaces in the viscous regularization of bilinear diffusion
  equations.
\newblock in preparation, 2023.

\bibitem[HN23]{HN23}
M.~Herrmann and B.~Niethammer.
\newblock Instability of hysteretic phase interfaces in a mean-field model with
  inhomogeneities.
\newblock {\em SIAM J. Appl. Math.}, 83(4):1422--1443, 2023.

\bibitem[H{\"o}l83]{Hol83}
K.~H{\"o}llig.
\newblock Existence of infinitely many solutions for a forward backward heat
  equation.
\newblock {\em Trans. Amer. Math. Soc.}, 278(1):299--316, 1983.

\bibitem[HPO04]{HPO04}
D.~Horstmann, K.~J. Painter, and H.~G. Othmer.
\newblock Aggregation under local reinforcement: from lattice to continuum.
\newblock {\em European J. Appl. Math.}, 15(5):546--576, 2004.

\bibitem[Jan23]{Jan23}
D.~Jan{\ss}en.
\newblock {\em {M}ehrskalendynamik hysteretischer {P}hasengrenzen in
  zeitdiskreten {V}orw\"arts-{R}\"uckw\"arts-{D}iffusionsgleichungen}.
\newblock PhD thesis, Technische Universit\"at Braunschweig, Department
  Mathematik, 2023.

\bibitem[LM12]{LM12}
P.~Lafitte and C.~Mascia.
\newblock Numerical exploration of a forward-backward diffusion equation.
\newblock {\em Math. Models Methods Appl. Sci.}, 22(6):1250004, 33, 2012.

\bibitem[MTT09]{MTT09}
C.~Mascia, A.~Terracina, and A.~Tesei.
\newblock Two-phase entropy solutions of a forward-backward parabolic equation.
\newblock {\em Arch. Ration. Mech. Anal.}, 194(3):887--925, 2009.

\bibitem[NCP91]{NCP91}
A.~Novick-Cohen and R.~L. Pego.
\newblock Stable patterns in a viscous diffusion equation.
\newblock {\em Trans. Amer. Math. Soc.}, 324(1):331--351, 1991.

\bibitem[Pad98]{Pad98}
V.~Padr\'{o}n.
\newblock Sobolev regularization of a nonlinear ill-posed parabolic problem as
  a model for aggregating populations.
\newblock {\em Comm. Partial Differential Equations}, 23(3-4):457--486, 1998.

\bibitem[Plo93]{Plo93}
P.~I. Plotnikov.
\newblock Equations with a variable direction of parabolicity and the
  hysteresis effect.
\newblock {\em Dokl. Akad. Nauk}, 330(6):691--693, 1993.

\bibitem[Plo94]{Plo94}
P.~I. Plotnikov.
\newblock Passing to the limit with respect to viscosity in an equation with
  variable parabolicity direction.
\newblock {\em Differential Eqns.}, 30(4):614--622, 1994.

\bibitem[PM90]{PM90}
P.~Perona and J.~Malik.
\newblock Scale-space and edge detection using anisotropic diffusion.
\newblock {\em IEEE Transactions on Pattern Analysis and Machine Intelligence},
  12(7):629--639, 1990.

\bibitem[Sma10]{Sma10}
F.~Smarrazzo.
\newblock Long-time behaviour of two-phase solutions to a class of
  forward-backward parabolic equations.
\newblock {\em Interfaces Free Bound.}, 12(3):369--408, 2010.

\bibitem[ST10]{ST10}
F.~Smarrazzo and A.~Tesei.
\newblock Long-time behavior of solutions to a class of forward-backward
  parabolic equations.
\newblock {\em SIAM J. Math. Anal.}, 42(3):1046--1093, 2010.

\bibitem[ST12]{ST13b}
F.~Smarrazzo and A.~Tesei.
\newblock Degenerate regularization of forward-backward parabolic equations:
  the regularized problem.
\newblock {\em Arch. Ration. Mech. Anal.}, 204(1):85--139, 2012.

\bibitem[ST13]{ST13a}
F.~Smarrazzo and A.~Terracina.
\newblock Sobolev approximation for two-phase \mbox{solutions} of
  forward-backward parabolic problems.
\newblock {\em Discrete Contin. Dyn. Syst.}, 33(4):1657--1697, 2013.

\bibitem[Ter11]{Ter11}
A.~Terracina.
\newblock Qualitative behavior of the two-phase entropy solution of a
  forward-backward parabolic problem.
\newblock {\em SIAM J. Math. Anal.}, 43(1):228--252, 2011.

\bibitem[Ter14]{Ter14}
A.~Terracina.
\newblock Non-uniqueness results for entropy two-phase solutions of
  forward-backward parabolic problems with unstable phase.
\newblock {\em J. Math. Anal. Appl.}, 413(2):963--975, 2014.

\bibitem[Ter15]{Ter15}
A.~Terracina.
\newblock Two-phase entropy solutions of forward-backward parabolic
  \mbox{problems} with unstable phase.
\newblock {\em Interfaces Free Bound.}, 17(3):289--315, 2015.

\bibitem[Vis06]{Vis06}
A.~Visintin.
\newblock Quasilinear parabolic {P}.{D}.{E}.s with discontinuous hysteresis.
\newblock {\em Ann. Mat. Pura Appl. (4)}, 185(4):487--519, 2006.

\end{thebibliography}
\end{document}